\documentclass{article}
\usepackage{amssymb,amsmath,amsthm}
\usepackage{amsfonts}
\usepackage{latexsym}
\usepackage{comment}
\newtheorem{The}{Theorem}[section]
\newtheorem{Lem}[The]{Lemma}
\newtheorem{Pro}[The]{Proposition}
\newtheorem{Rem}[The]{Remark}
\newtheorem{Cor}[The]{Corollary}

\newtheorem{Def}[The]{Definition}
\numberwithin{equation}{section}
\newcommand{\RR}{\mathbb R}
\newcommand{\QQ}{\mathbb Q}
\newcommand{\ZZ}{\mathbb Z}
\newcommand{\TT}{\mathbb T}
\newcommand{\Aa}{\mathbb A}
\newcommand{\RB}{\mathcal{R}}
\newcommand{\CB}{\widetilde{\mathcal{R}}}
\newcommand{\KK}{\mathcal{K}}
\newcommand{\DD}{\mathcal{D}}

\def\dx{\dot{x}}

\def\ddx{\ddot{x}}

\def\ddr{\ddot{r}}
\def\dr{\dot{r}}
\def\dt{\dot{\theta}}
\def\th{\hat{\tau}}
\def\nh{\hat{\eta}}
\def\dR{\dot{R}}
\def\ddR{\ddot{R}}

\def\eps{\varepsilon}

\usepackage{xcolor,fullpage}

\begin{document}

\title{Chaotic motion in the breathing circle billiard}
\author{Claudio Bonanno\thanks{
Dipartimento di Matematica, Universit\`a di Pisa, Largo Bruno 
Pontecorvo 5, 56127 Pisa, Italy. E-mail: 
\texttt{claudio.bonanno@unipi.it}.}\and Stefano Mar\`o\thanks{
Dipartimento di Matematica, Universit\`a di Pisa, Largo Bruno 
Pontecorvo 5, 56127 Pisa, Italy. E-mail: 
\texttt{stefano.maro@unipi.it}.}}
\date{}
\maketitle

\abstract{
We consider the free motion of a point particle inside a circular billiard with periodically moving boundary, with the assumption that the collisions of the particle with the boundary are elastic so that the energy of the particle is not preserved. It is known that if the motion of the boundary is regular enough then the energy is bounded due to the existence of invariant curves. We show that it is nevertheless possible that the motion of the particle is chaotic, also under regularity assumptions for the moving boundary. More precisely, we show that there exists a class of functions describing the motion of the boundary for which the billiard map admits invariant probability measures with positive metric entropy. The proof relies on variational techniques based on Aubry-Mather theory.
}


\bigskip
\textbf{2020MSC:} 37C83, 37B40, 37E40 

\section{Introduction}

A mathematical billiard with moving boundary is a region of the plane instantaneously bounded by a closed curve changing with time. The billiard problem then consists of the free motion of a point particle inside this region colliding elastically with the moving boundary.

The study of time-dependent billiards can be motivated physically by the study of confined Lorentz gas or by some models in nuclear physics (see for instance \cite{blocki_et_al,burgio_et_al,glanz,pais}). In the physical contexts, a relevant question is whether the elastic bounces can make the energy of the particle grow infinitely. The question was raised by Fermi \cite{fermi} trying to explain the high velocity gained by photons. A mathematical formulation of the problem was proposed by Ulam and is now called the Fermi-Ulam model. It describes the free motion of a particle between two parallel walls moving periodically. ``Fermi acceleration'' then occurs if the elastic bounces with the moving walls make the energy of the particle tend to infinity. It is known that it depends on the regularity of the motion of the walls. Actually, it was proved in \cite{levi} that if the motion of the wall is at least $C^6$ then KAM theory applies and the energy remains bounded. On the other hand, in \cite{zar} it is shown how to construct a motion of the walls that is only continuous and allows the energy to grow to infinity (see also \cite{maro4} for a similar result in a related impact model).   

Time-dependent billiards can be seen as natural generalisations of the Fermi-Ulam model and the question on the existence of Fermi acceleration naturally extends to this context. In this case, the answer also depends on the geometry of the boundary.  

If the boundary is a moving ellipse then it has been proved in \cite{dettman_ellipse} that it is possible to construct orbits that gain energy. Moreover, the existence of unbounded motions is a symptom of complex dynamics. In fact, in \cite{dettman_ellipse} it is also proved that the phenomenon of splitting of the separatrices occurs and a scattering map can be defined. We refer to \cite{turaev_chaos} for more insight on the topic of Fermi acceleration in general time-dependent billiards.

Instead, if the boundary is a circle with radius varying smoothly with time then KAM theory applies and the energy remains bounded \cite{sylvie1}. Since the motion along a diameter is described by the classical Fermi-Ulam model, the counterexample constructed in \cite{zar} shows that the regularity of the motion of the boundary is a fundamental assumption in \cite{sylvie1}.

In this paper we deal with the case of a region with circular boundary of radius $R(t)$ periodic in time. The region is called the breathing circle billiard. We show that, even if the motion of the boundary is regular, chaotic phenomena can occur. More precisely, we find a class of function $\CB$ such that if $R\in\CB$ then there exist many invariant measures of the associated billiard map with positive metric entropy. The class $\CB$ has a somehow technical definition but it can be shown that a representative is
\[
R(t) =M+\delta\, \sin(2\pi t)
\]
with $M$ sufficiently large with respect to $\delta$.

The dynamics of a time-dependent billiard whose boundary remains a convex curve can be described by a 4-dimensional exact symplectic map \cite{koiller_et_al}. However, for the breathing circle billiard, the angular momentum is a first integral so that the dynamics can be reduced to a two dimensional map of the cylinder that turns out to be exact symplectic and twist. We show that the reduced map enters in the variational framework of Aubry-Mather theory, which implies the existence of interesting invariant sets. In particular a key role is played by invariant curves with irrational rotation numbers. More precisely, it is known that the lack of invariant curves for a given irrational rotation number implies the existence of chaotic motion (see \cite{angenent1,angenent2,forni}).
In the last decades many results have been proved in the direction of ``breaking'' invariant curves, also in higher dimensions, giving rise to the so called ``converse KAM'' theory \cite{haro,mkmeiss,percivalmk}.    

In this paper we use a criterion based on the variational approach of Aubry-Mather theory in the spirit of what is done in \cite{maro5}. More precisely, it is known that orbits of exact symplectic twist maps correspond to stationary points of an action, and the ones on invariant curves are minimal. As a consequence, the second variation of the action must be positive on orbits lying on invariant curves. The main idea to prove our main result is then to show that if $R\in\CB$ the second variation of the action is negative in a zone of the phase space, preventing the existence of invariant curves for some irrational rotation numbers. From a technical point of view, in order to compute the second variation of the action, one needs the generating function of the associated diffeomorphism. A considerable part of the paper is dedicated to get an explicit formulation of the generating function of the billiard map. To this aim we follow the idea, used in \cite{kunzeortega} for the (non-periodic) Fermi-Ulam model, that the generating function is given by the Lagrangian action of a solution of the Dirichlet problem between two consecutive impacts. We conclude noting that a consequence of our approach is the existence of Aubry-Mather sets with different rotation numbers, giving rise to periodic and quasi-periodic motions of the breathing circle billiard. Similar results have been obtained for other systems with impacts such as bouncing balls \cite{maro3}.

The paper is organised as follows. In Section \ref{sec:state} we state the problem and the main results of the paper. In Section \ref{sec:dir} we study the Dirichlet problem between two consecutive impacts, and its results are used in Section \ref{sec:genfun} to compute the generating function of the billiard map. In Section \ref{sec:per} we describe the periodic and quasi-periodic motions. Section \ref{sec:chaos} is dedicated to the proof of the main theorem on chaotic motion. Appendix \ref{technical} contains the proof that the class of functions $\CB$ is not empty. Finally, the main results of Aubry-Mather theory used in the paper are collected in Appendix \ref{sec:am}.

\bigskip
\textbf{Acknowledgements:} This research was partially supported by the PRIN Grant 2017S35EHN of the Ministry of Education, University and Research (MIUR), Italy. It   is   also   part   of   the   authors’   activity   within   the   DinAmicI   community (\texttt{www.dinamici.org}) and  the  Gruppo  Nazionale  di  Fisica  Matematica,  INdAM, Italy.

\section{Statement of the problem and main results}\label{sec:state}

Let $R(t)$ be a strictly positive function and let the breathing circle $\DD_t$ be the bounded region of the plane with moving boundary $\partial \DD_t =\{x\in\RR^2 \: : \: |x|^2 = R^2(t) \}$. Let us consider a particle of unitary mass moving freely inside $\DD_t$ and satisfying the elastic impact law at every bounce on the boundary. Assume that $\partial \DD_t $ is positively oriented and denote by $\th$, $\nh$ the unitary tangent and outward normal vectors at points of $\partial \DD_t $. By a bouncing solution we mean a continuous function
\[
x : \RR\to \RR^2
\]
with a sequence of impact times $(t_n)_{n\in\ZZ}$ satisfying
\begin{enumerate}
\item $\ddx(t) =0$ for $t\in(t_n,t_{n+1})$ for every $n\in\ZZ$,
\item $|x(t)|< R(t) $ for $t\in(t_n,t_{n+1})$ for every $n\in\ZZ$,
\item $|x(t_n)| = R(t_n)$ for every $n\in\ZZ$,
\item $\dx(t_n^+)\cdot \th_n =\dx(t_n^-)\cdot \th_n$ and $\dx(t_n^+)\cdot \nh_n =-\dx(t_n^-)\cdot \nh_n + 2\dR(t_n)$, where $\th_n$, $\nh_n$ are the unitary tangent and outward normal vectors at $x(t_n)$, and $\dx(t_n^-)$ and $\dx(t_n^+)$ denote the velocity vector just before and after the bounce at time $t_n$ respectively.  
 \end{enumerate}

Condition $4$ describes the elastic bouncing condition: the tangent component of the velocity is preserved and the impulse is given in the normal direction. Note that if $\dR=0$ we get the usual mirror law.  

\begin{Pro}
Let $x(t)$ be a bouncing solution with impact times $(t_n)_{n\in\ZZ}$, then the angular momentum  $C(t) = x(t)\wedge \dx(t)$ is constant for every $t$.
\end{Pro}
\begin{proof}
It is clear that the angular momentum is constant for $t\in(t_n,t_{n+1})$ for every $n\in\ZZ$. Moreover at the bouncing time it holds
\begin{align*}
    C(t_n^-) &= x(t_n)\wedge \dx(t_n^-)= R(t_n)\nh_n \wedge \dx(t_n^-)=R(t_n)\nh_n \wedge (\dx(t_n^-)\cdot \th_n) \th_n  \\
    & =R(t_n)\nh_n \wedge (\dx(t_n^+)\cdot\th_n) \th_n = R(t_n)\nh_n \wedge \dx(t_n^+) = C(t_n^+)    
\end{align*}
hence the proposition is proved.
\end{proof}
Since the motion is in the plane we have $C(t)=(0,0,c)$. If $c=0$ the motion is along a diameter and never leaves the diameter. Moreover there is a symmetry between motion in the clockwise and in the anticlockwise direction given by changing sign to $c$ (see \eqref{dirich_r}). Without loss of generality in the following we assume $c>0$, which corresponds to anticlockwise motion.

To state the main result, we introduce two classes of functions.
\begin{Def}\label{def_R}
Let $R(t)$ be a $C^2$, strictly positive and $1$-periodic function. We denote by $\|\cdot\|$ the sup-norm of a function and use the notation  
\[
\underline{R} := \min\limits_t\, R(t)\, , \qquad \overline{R}:= \max\limits_t\, R(t)
\]
and 
\[
\sigma := \min \left\{ \frac{\underline{R}}{2\, \| \dR \|}\, ,\, \frac{2\, \sqrt{1+\sqrt{1-\eps^2}}\, \underline{R}}{\sqrt{\| \frac{d^2}{dt^2} R^2 \|}} \right\}
\]
with $\eps\in(0,1)$ a fixed parameter.

We say that $R(t)$ belongs to the class $\RB$ if $\sigma>2$.

We say that $R(t)$ belongs to the class $\CB$ if:
\begin{itemize}
\item[(i)] $\sigma>4$;
\end{itemize}
and there exists $\bar{t}\in [0,1)$ such that $(\ddR(\bar{t})\underline{R}+\| \dR\| \overline{R})<0$ and
\begin{itemize}
\item[(ii)] \[
3<1+\sqrt{\frac{2\overline{R}^2}{-\ddR(\bar{t})\underline{R}-\| \dR\| \overline{R}}}<-1+\sqrt{\frac{2\underline{R}^2}{\frac{2\overline{R}^2}{\sigma^2}+\| \dR\| \overline{R}}}\, ;
\]
\item[(iii)]
\[
\dR(\bar{t})=0 \quad\mbox{and}\quad \ddR(\bar{t})<-\frac{2\overline{R}^2}{\sigma^2\underline{R}}\, . 
\]
\end{itemize}
\end{Def}  

For simplicity we drop the dependence on $\eps$ in the notations for $\sigma$ and the classes $\RB$ and $\CB$. Moreover the parameter $\eps$ is to be considered fixed in $(0,1)$ for the rest of the paper, there isn't a more interesting value for it. Also note that clearly $\CB\subset\RB$, moreover the classes $\RB$ and $\CB$ are non empty as shown in the following Proposition, whose technical proof is in Appendix \ref{technical}. 

\begin{Pro}\label{R_not_empty}
For a fixed $\eps\in (0,1)$ let $\alpha:=\sqrt{1+\sqrt{1-\eps^2}}$ and
\[
\bar k:= \frac{\alpha^2+\sqrt{2\alpha^4-1}}{\alpha^2-1}\, .
\]
For every integer $k>\bar k$ and $\delta>0$ such that
\[
\frac{1}{4\pi^2(k^2+1)}<\delta<\frac{1}{2\pi(k+1)}
\]
there exists $M>0$ such that the function
\[
R_{k,\delta,M}(t) := M +\delta\sin\left(2\pi k t\right) +\delta\sin(2\pi t) 
\]
belongs to the class $\CB$.

Moreover, if $\eps<\sqrt{1-\frac{1}{(\pi-1)^2}}$, then there exist $\delta,M$ such that 
\[
R_{\delta,M}(t) := M +\delta\sin(2\pi t)
\]
is in $\CB$.
\end{Pro}

\begin{Rem} \label{rem:choice-R}
It is interesting for the conclusions of Theorem \ref{Main1} and for Remark \ref{rem2} to notice that the functions $R_{k,\delta,M}(t)$ and $R_{\delta,M}(t)$ are 1-periodic and actually 1 is their least period. 
\end{Rem}

\begin{Rem} \label{rem:epsilon}
Here we add some comments on the parameter $\eps$. Let us first note that if $\eps=1$ there are problems in the proof of Proposition \ref{R_not_empty}, in particular condition \eqref{serve_eps} does not hold since we would have $\alpha=1$. More in general, for $\eps=1$, it is not clear if the class $\CB$ is not empty. From condition $(iii)$ and the definition of $\sigma$,
\[
\ddR(\bar{t})<-\frac{2\overline{R}^2}{\sigma^2\underline{R}}\leq
-\frac{\overline{R}^2}{2\underline{R}^3}\left\| \frac{d^2}{dt^2} R^2 \right\|\, .
\]
But
\[
\left\| \frac{d^2}{dt^2} R^2 \right\| \ge \left| \frac{d^2}{dt^2} R^2 \Big|_{t=\bar t} \right| = 2 \overline{R}\, |\ddR(\bar{t})|
\]
if $R(\bar t)=\overline{R}$, so that
\[
\ddR(\bar{t})<- \frac{\overline{R}^3}{\underline{R}^3}\, |\ddR(\bar{t})|
\]
that is impossible.
\end{Rem}

We now begin to state our main results. The first concerns the existence of regular motion for the billiard dynamics inside $\DD_t$. 

\begin{The}\label{Main1}
Suppose that  $c\in \left(0,\eps\, \frac{\underline{R}^2}{\sigma}\right)$ and let $R(t)\in\RB$. Then, for every $1<\omega<\sigma-1$ 
there exists a family of bouncing solutions
    \[
  \left\{x_\xi(t)\right\}_{\xi\in\RR}
  \]
  such that for every $\xi\in\RR$
\begin{align}\label{contheo} 
& x_{\xi+1}(t) = x_{\xi}(t-1),   \\ 
\label{contheo2}
& x_{\xi}(t)=x_{\xi+\omega}(t).
\end{align}
Moreover, the sequence of impact times $(t_n)_{n\in\ZZ}$ satisfies
\[
\lim_{n\to\infty}\frac{t_n}{n} =\omega
\]
and, if $\omega=p/q\in\QQ$, then $t_{n+q}=t_n+p$ for every $n\in\ZZ$. 
\end{The}

\begin{Rem}\label{rem2}
If $\omega=p/q\in\mathbb{Q}$, then
  \[
x_{\xi}(t+p)=x_{\xi-p}(t)=x_{\xi-q\frac{p}{q}}(t)=x_{\xi}(t)
  \]
so that the solution makes $q$ bounces in time $p$ before repeating itself. They are said $(p,q)$-{\it periodic}. 
  
If $\omega\in\RR\setminus\mathbb{Q}$, solutions satisfying \eqref{contheo}-\eqref{contheo2} can be seen as generalised quasi-periodic. Actually, consider the function
  \[
  \Phi_\xi(a,b)= x_{\xi-\omega b+a}(a).
  \]
  This function is doubly-periodic in the sense that
  \begin{align*}
    \Phi_\xi(a+1,b)&= x_{\xi-\omega b+a+1}(a+1)=\Phi_\xi(a,b), \\
    \Phi_\xi(a,b+1)&=x_{\xi-\omega b+a-\omega}(a) =\Phi_\xi(a,b)
  \end{align*}
  and
  $\Phi_\xi(t, t/\omega)=x_{\xi}(t)$. 
  If the function $\xi\mapsto\Phi_\xi$ is continuous, then these solutions are classical quasi-periodic solutions with frequencies $(1,1/\omega)$ in the sense of \cite{SM} (see also \cite{Ortega2}). We will not guarantee the continuity, however, the function $\xi\mapsto\Phi_\xi$ will have at most jump discontinuities and if $\xi$ is a point of continuity then so are $\xi+\omega,\xi+1$. In this latter case the solutions are known as generalised quasi-periodic. Under the hypothesis of Theorem \ref{Main1} we are not able to distinguish between the two classes of solutions. However, if $R(t)\in C^9$, then the results in \cite{sylvie1} implies the existence of classical quasi-periodic solutions for some $\omega$. On the other hand, as a byproduct of our next result, if $R(t)\in\CB$ then generalised quasi-periodic solutions exist.     \end{Rem}

\begin{Rem} \label{rem:omega}
The restriction $\omega>1$ is not optimal and is due to the technique used in the proof. We cannot guarantee the conjecture that quasi-periodic solutions exist for every $\omega>0$. However, this is not the main purpose of this paper. Actually, we will find chaotic dynamics corresponding (in a sense that will be specified) to $\omega>1$.    
\end{Rem}  
  
Our second main result concerns the existence of chaotic dynamics for the billiard motion in $\DD_t$. A detailed statement is given in Theorem \ref{Main2prime} of Section \ref{sec:chaos}.

\begin{The}\label{Main2}
Suppose that $R(t)\in\CB$. Then the billiard map admits many invariant probability measures with positive metric entropy.     
\end{The}

\section{The Dirichlet problem}\label{sec:dir}

The proofs of our main results rely on Aubry-Mather theory, for which we need to define a generating function for the billiard map associated to the billiard flow in $\DD_t$. The first step is the study of the Dirichlet problem associated to the flow between two consecutive bounces. 

Let us consider two consecutive bounces for the billiard motion in $\DD_t$ at times $t_n$ and $t_{n+1}$, the Dirichlet problem which describes the billiard motion between the two bounces is
\[
  \left\{
  \begin{aligned}
    & \ddot{x}(t) = 0, \quad t\in (t_n,t_{n+1}), \\
    & |x(t)| < R(t), \quad t\in (t_n,t_{n+1}), \\
   & |x(t_n)|=R(t_n), \\
   & |x(t_{n+1})|=R(t_{n+1})
\end{aligned}
\right.
\]
and in polar coordinates $(r, \theta)$ it transforms into
\begin{equation}
    \label{dirich_r}
  \left\{
  \begin{aligned}
    & \ddr = \frac{c^2}{r^3}, \quad t\in (t_n,t_{n+1}), \\
    & r^2\dt =c, \quad t\in (t_n,t_{n+1}), \\
    & r(t) < R(t), \quad t\in (t_n,t_{n+1}), \\
   & r(t_n)=R(t_n), \\
   & r(t_{n+1})=R(t_{n+1})
\end{aligned}
\right.
  \end{equation}
from which it is evident the symmetry with respect to the change of sign of $c$. Bouncing condition 4 in these coordinates reads
  \begin{align}
    \label{tgcondr}
    &\dt(t_n^+) =\dt(t_n^-), \\
    \label{normcondr}
      &\dr(t_n^+) =-\dr(t_n^-) + 2\dR(t_n).
\end{align}

In the following result we find sufficient conditions for a solution of system \eqref{dirich_r} to exist. Note that these solutions do not satisfy in general the bouncing conditions \eqref{tgcondr}-\eqref{normcondr} when glued together.  
  
\begin{Pro}\label{dirich_sol}
Let $\eps \in (0,1)$ be a fixed parameter, and let us fix a value of $c>0$. For all times $t_n,t_{n+1}$ satisfying
\begin{align}\label{hp_dirsol1}
  & 0<t_{n+1}-t_n<\eps\, \frac{\underline{R}^2}{c}, \\
  \label{hp_dirsol2}
  & 0<t_{n+1}-t_n< \frac{\underline{R}}{2\, \| \dR \|}, \\
  \label{hp_dirsol3}
  & 0<t_{n+1}-t_n< \frac{2\, \sqrt{1+\sqrt{1-\eps^2}}\, \underline{R}}{\sqrt{\| \frac{d^2}{dt^2} R^2 \|}}
\end{align}
system \eqref{dirich_r} admits a solution $(r(t;t_n,t_{n+1}),\theta(t;t_n,t_{n+1}))$ such that
\begin{equation} \label{cons-bound}
\dr(t_n^+)< \min\{0,\dR(t_n)\}\, , \qquad \dr(t_{n+1}^-)>\max\{0,2\dR(t_{n+1})\}
\end{equation}
\end{Pro}
 
\begin{proof}
  Define $R_n=R(t_n)$, $R_{n+1}=R(t_{n+1})$ and $\tau_n = t_{n+1}-t_n$.
  
The equation $\ddr = c^2/r^3$ has the first integral
\begin{equation} \label{energy}
A =2E = \dr^2 + \frac{c^2}{r^2} >0
\end{equation}
where $E$ is the energy of the system, and can be integrated giving the general solution
\begin{equation} \label{def_r}
r(t) = \sqrt{\frac{c^2+A^2(t+B)^2}{A}}
\end{equation}
for $B\in\RR$. System \eqref{dirich_r} is rotationally invariant so that we can fix $\theta(t_n)=0$ without loss of generality. With this assumption the solution \eqref{def_r} represents a straight trajectory in the billiard table starting from the point $(R_n,0)$ with velocity $\dx(t_n)$ satisfying
\begin{equation} \label{vel-AB}
\dx(t_n) = \begin{pmatrix}
\frac{A(t_n+B)}{R_n}\\[0.2cm] \frac{c}{R_n}
\end{pmatrix} \quad \text{and}\quad  |\dx(t_n)|^2 = 2E = A.
\end{equation}
Moreover notice that by \eqref{def_r}
\begin{equation}\label{vel-rad-AB}
\dr(t) =\frac{A(t+B)}{r(t)}.
\end{equation}

Let $\ell := A(t_n+B)$, then the straight trajectory \eqref{def_r} is parametrised by
\[
[0,+\infty) \ni s \, \mapsto \, \begin{pmatrix} R_n + s\, \frac{\ell}{R_n}\\[0.2cm] s\, \frac{c}{R_n} \end{pmatrix}
\]
and the condition $r(t_{n+1})=R_{n+1}$ is then equivalent to
\[
\Big(R_n + \tau_n\, \frac{\ell}{R_n}\Big)^2 + \tau_n^2\, \frac{c^2}{R_n^2} = R_{n+1}^2.
\]
We then obtain
\begin{equation}\label{sonoinr1}
\ell_\pm = \frac{R_n}{\tau_n} \, \left(-R_n \pm \sqrt{R_{n+1}^2 -\tau_n^2\, \frac{c^2}{R_n^2}} \right)
\end{equation}
which is well defined thanks to \eqref{hp_dirsol1}. Using \eqref{vel-AB} it also holds
\begin{equation}\label{AB-l}
A_\pm = \frac{\ell_\pm^2+c^2}{R_n^2}\, , \quad B_\pm=\frac{R_n^2\, \ell_\pm}{\ell_\pm^2+c^2}
\end{equation}
which in \eqref{def_r} gives $r(t_n)=R_n$. We still have to verify conditions \eqref{cons-bound} and $r(t)<R(t)$ for all $t\in (t_n,t_{n+1})$.

First we consider the case $R_{n+1}\ge R_n$. From \eqref{vel-rad-AB} for the first condition in \eqref{cons-bound} to hold we need
\[
\dr(t_n^+) = \frac{\ell_\pm}{R_n} < 0.
\]
The solution $\ell_-$ is always admissible, but solution $\ell_+$ is admissible only for
\[
\tau_n > \frac{R_n}{c}\, \sqrt{R_{n+1}^2-R_n^2}.
\]
Hence we exclude $\ell_+$. If $\dR(t_n)\ge 0$, the choice $\ell_-$ implies that the first condition in \eqref{cons-bound} is satisfied. If $\dR(t_n)<0$ we need to check that
\[
\dr(t_n^+) = \frac{\ell_-}{R_n} < \dR(t_n).
\]
We have
\[
\ell_- < R_n\, \dR(t_n) \quad \Leftrightarrow \quad  \sqrt{R_{n+1}^2 -\tau_n^2\, \frac{c^2}{R_n^2}} > - R_n-  \dR(t_n)\tau_n
\]
which is satisfied by \eqref{hp_dirsol2} since
\[
\tau_n <  \frac{\underline{R}}{2\, \| \dR \|} \le \frac{R_n}{2\, |\dR(t_n)|} < \frac{R_n}{|\dR(t_n)|}
\]
implies $(-\dR(t_n)\tau_n - R_n) <0$. The first condition in \eqref{cons-bound} is then satisfied.

Let us now consider the second condition in \eqref{cons-bound}. From \eqref{vel-rad-AB} first we need to impose
\[
\dr(t_{n+1}^-) = \frac{A(t_{n+1}+B)}{R_{n+1}}  = \frac{A\tau_n +\ell_-}{R_{n+1}}>0.
\]
By \eqref{AB-l} we have
\[
A\tau_n +\ell_- = \frac{\ell_-^2+c^2}{R_n^2}\, \tau_n +\ell_- =
\]
\[
= \frac{1}{\tau_n} \left[ \left(R_n +\sqrt{R_{n+1}^2 -\tau_n^2\, \frac{c^2}{R_n^2}}\right)^2 + \tau_n^2\, \frac{c^2}{R_n^2} - R_n\, \left(R_n +\sqrt{R_{n+1}^2 -\tau_n^2\, \frac{c^2}{R_n^2}}\right)\right] =
\]
\[
= \frac{1}{\tau_n} \left( R_{n+1}^2 +R_n\, \sqrt{R_{n+1}^2 -\tau_n^2\, \frac{c^2}{R_n^2}}\right) >0\, .
\]
If $\dR(t_{n+1})<0$ we are done, on the contrary we need to verify
\[
\dr(t_{n+1}^-) = \frac{A(t_{n+1}+B)}{R_{n+1}}  = \frac{A\tau_n +\ell_-}{R_{n+1}}> 2\dR(t_{n+1}).
\]
As before this is equivalent to
\[
\frac{1}{\tau_n} \left( R_{n+1}^2 +R_n\, \sqrt{R_{n+1}^2 -\tau_n^2\, \frac{c^2}{R_n^2}}\right) > 2\, R_{n+1}\, \dR(t_{n+1}).
\]
By \eqref{hp_dirsol2} we have $\frac{R_{n+1}^2}{\tau_n} > 2 R_{n+1} \dR(t_{n+1})$, hence we are done also with the second condition in \eqref{cons-bound} since
\[
\frac{1}{\tau_n} \left( R_{n+1}^2 +R_n\, \sqrt{R_{n+1}^2 -\tau_n^2\, \frac{c^2}{R_n^2}}\right) > \frac{R_{n+1}^2}{\tau_n}. 
\]
It remains to show that $r(t)<R(t)$ for all $t\in (t_{n},t_{n+1})$. From \eqref{cons-bound} it follows that
\[
\frac{d}{dt} \Big(R^2(t)-r^2(t)\Big)|_{t=t_n^+} > 0\, , \quad \frac{d}{dt} \Big(R^2(t)-r^2(t)\Big)|_{t=t_{n+1}^-} < 0. 
\]
Moreover from \eqref{vel-rad-AB} and \eqref{AB-l}

\[
\frac{d^2}{dt^2} \Big(R^2(t)-r^2(t)\Big) = \Big(\frac{d^2}{dt^2} R^2(t) \Big) - 2A =
\]
\[
= \Big(\frac{d^2}{dt^2} R^2(t) \Big) - \frac{2}{\tau_n^2}\, \left( R_n^2+R_{n+1}^2 +2R_n\, \sqrt{R_{n+1}^2 -\tau_n^2\, \frac{c^2}{R_n^2}}\right) <
\]
\[
< \left\| \frac{d^2}{dt^2} R^2 \right\| -  \frac{2}{\tau_n^2}\, \left( 2\, \underline{R}^2 + 2\underline{R} \, \sqrt{\underline{R}^2(1-\eps^2)}\right) 
\]
where in the last inequality we have used \eqref{hp_dirsol1}. It follows that
\[
\frac{d^2}{dt^2} \Big(R^2(t)-r^2(t)\Big) < \left\| \frac{d^2}{dt^2} R^2 \right\| - \frac{4(1+\sqrt{1-\eps^2})\, \underline{R}^2}{\tau_n^2} <0
\]
by \eqref{hp_dirsol3}. Hence, since $r^2(t_n)=R^2(t_n)$ and $r^2(t_{n+1})=R^2(t_{n+1})$, we have that $R^2(t)>r^2(t)$ for all $t\in (t_n,t_{n+1})$. This concludes the proof of the proposition in the case $R_{n+1}\ge R_n$, for which the solution is given by \eqref{def_r} with $A,B$ as in \eqref{AB-l} with $\ell=\ell_-$ as defined in \eqref{sonoinr1}.

The case $R_{n+1}<R_n$ is similar. The condition $\dr(t_n^+)<0$ in this case is verified for both $\ell_\pm$ for all $\tau_n$ verifying \eqref{hp_dirsol1}. But if $\dR(t_n)<0$ we also need to verify $\dr(t_n^+)<\dR(t_n)$ which is equivalent to
\[
\ell_\pm < R_n\, \dR(t_n) \quad \Leftrightarrow \quad \pm \sqrt{R_{n+1}^2 -\tau_n^2\, \frac{c^2}{R_n^2}} < R_n + \dR(t_n)\tau_n.
\]
It is clear that for $\ell=\ell_+$ more conditions on $\tau_n$ may be necessary. On the contrary for $\ell=\ell_-$ all the arguments used in the case $R_{n+1}\ge R_n$ can be repeated, and as above the proposition is proved with the solution given by \eqref{def_r} with $A,B$ as in \eqref{AB-l} with $\ell=\ell_-$ as defined in \eqref{sonoinr1}.

In conclusion, we have proved that for all times $t_n,t_{n+1}$ satisfying \eqref{hp_dirsol1}-\eqref{hp_dirsol3}, the functions
\begin{equation}\label{def-r-finale}
\begin{aligned}
& r(t;t_n,t_{n+1}) := \sqrt{\frac{A^2(t_n,t_{n+1})\, (t+B(t_n,t_{n+1}))^2+c^2}{A(t_n,t_{n+1})}},\\[0.2cm]
& \theta(t;t_n,t_{n+1}) := \theta(t_n) + \int_{t_n}^{t_{n+1}}\, \frac{c}{r^2(t;t_n,t_{n+1})}\, dt
\end{aligned}
\end{equation}
where $\theta(t_n)$ is the angular variable of the bounce point at time $t_n$, are a solution to system \eqref{dirich_r} satisfying also the conditions \eqref{cons-bound} by letting
\begin{equation}\label{def-AB-finale}
\begin{aligned}
& A(t_n,t_{n+1}) = \frac{R^2(t_n)+R^2(t_{n+1}) + 2\sqrt{R^2(t_n)R^2(t_{n+1})-c^2(t_{n+1}-t_n)^2}}{(t_{n+1}-t_n)^2},\\[0.2cm]
& B(t_n,t_{n+1}) = - \, \left(t_n + \frac{R^2(t_n)+ \sqrt{R^2(t_n)R^2(t_{n+1})-c^2(t_{n+1}-t_n)^2}}{ (t_{n+1}-t_n)\, A(t_n,t_{n+1})}\right).
\end{aligned}
\end{equation}
\end{proof}

\begin{Rem}\label{rem-unire}
Given times $t_n, t_{n+1}$ and $t_{n+2}$ such that the intervals $(t_{n+1}-t_n)$ and $(t_{n+2}-t_{n+1})$ satisfy assumptions \eqref{hp_dirsol1}-\eqref{hp_dirsol3}, the solutions $r(t;t_n,t_{n+1})$ and $r(t;t_{n+1},t_{n+2})$ given in \eqref{def-r-finale} are compatible with the bouncing condition \eqref{normcondr}, in the sense that if $r(t;t_n,t_{n+1})$ satisfies \eqref{cons-bound} at $t=t_{n+1}$ then
\[
\dr(t_{n+1}^+;t_{n+1},t_{n+2}) := - \dr(t_{n+1}^-;t_n,t_{n+1}) + 2\dR(t_{n+1}) < \min\{0,\dR(t_{n+1})\}
\]
that is also $r(t;t_{n+1},t_{n+2})$ satisfies \eqref{cons-bound} at $t=t_{n+1}$.
\end{Rem}

\begin{Rem}\label{rem-angolo}
Given $\theta(t_n)$ and $\theta(t_{n+1})$ the angular variables of two consecutive bounces at times $t_n$ and $t_{n+1}$ determined by the solution \eqref{def-r-finale}, the angular variation $\delta\theta(t_n,t_{n+1}) := \theta(t_{n+1})-\theta(t_n)$ is given by
\begin{equation}\label{ang-variation}
\delta\theta(t_n,t_{n+1}) = \pi - \arctan \left( \frac{c\, (t_{n+1}-t_n)}{\sqrt{R^2(t_n)R^2(t_{n+1})-c^2(t_{n+1}-t_n)^2}} \right).
\end{equation}
Indeed, the length of the trajectory between the two bounces is given by $(t_{n+1}-t_n) |\dot x(t_n)| = (t_{n+1}-t_n) \sqrt{A(t_n,t_{n+1})}$ (see \eqref{vel-AB}), hence by Carnot's Theorem
\[
(t_{n+1}-t_n)^2 A(t_n,t_{n+1}) = R^2(t_n)+R^2(t_{n+1}) - 2\, R(t_n)\, R(t_{n+1})\, \cos \delta\theta(t_n,t_{n+1})
\]
and by \eqref{def-AB-finale}
\[
\cos \delta\theta(t_n,t_{n+1}) = -\sqrt{1-\frac{c^2 (t_{n+1}-t_n)^2}{ R^2(t_n)\, R^2(t_{n+1}) }}.
\]
Then \eqref{ang-variation} immediately follows by computing $\sin  \delta\theta(t_n,t_{n+1})$ using that $c>0$, so that $ \delta\theta(t_n,t_{n+1}) \in (0,\pi)$.
\end{Rem}

\section{The generating function}\label{sec:genfun}

In this section we define a generating function $h(t_0,t_1)$ for the billiard map between two consecutive bounces on $\partial \DD_t$ at times $t_0$ and $t_1$.

Following the approach in \cite{kunzeortega}, we first show that a good guess for $h$ is given by the action of system \eqref{dirich_r}. Let 
\begin{equation}
  \label{genfun}
h^c(t_n,t_{n+1}) =\int_{t_n}^{t_{n+1}} L^c(r(t;t_n,t_{n+1}),\dr(t;t_n,t_{n+1})) dt
  \end{equation}
where $r(t;t_n,t_{n+1})$ is the solution to \eqref{dirich_r} found in Proposition \ref{dirich_sol} and
\begin{equation}
  \label{def_L}
L^c(r,\dr) = \frac{1}{2}\dr^2-\frac{c^2}{2r^2}
\end{equation}
is the reduced Lagrangian of the system for a fixed value $c>0$. In particular, $r(t;t_n,t_{n+1})$ satisfies the Euler-Lagrange equation
\[
\frac{d}{dt}\left(\frac{\partial L^c}{\partial\dr}\right) = \frac{\partial L^c}{\partial r}.
\]
For simplicity, let us denote $r(t) = r(t;t_n,t_{n+1})$ and remove the dependence on $c$. We have
\[
  \partial_{t_n} h(t_n,t_{n+1}) = \partial_{t_n}\int_{t_n}^{t_{n+1}} L(r(t),\dr(t)) dt =
  -L(r(t_n),\dr(t_n^+)) + \int_{t_n}^{t_{n+1}} \frac{\partial L}{\partial \dr}\frac{\partial \dr}{\partial t_n} + \frac{\partial L}{\partial r}\frac{\partial r}{\partial t_n}dt=
\]
\[
= -L(r(t_n),\dr(t_n^+)) + \left[ \frac{\partial L}{\partial \dr}\frac{\partial r}{\partial t_n}   \right]_{t=t_n}^{t=t_{n+1}}+ \int_{t_n}^{t_{n+1}} \left[-\frac{d}{dt}\left(\frac{\partial L}{\partial \dr}\right) + \frac{\partial L}{\partial r}\right]\frac{\partial r}{\partial t_n}dt =
\]
\[
= -L(r(t_n),\dr(t_n^+)) +  \frac{\partial L}{\partial \dr}(t_{n+1})\frac{\partial r}{\partial t_n}(t_{n+1}) - \frac{\partial L}{\partial \dr}(t_n)\frac{\partial r}{\partial t_n}(t_{n}), 
\]
where we first used integration by parts, then the fact that $r(t)$ satisfies the Euler-Lagrange equation and finally denoted
\[
\frac{\partial L}{\partial \dr}(t_{n+1})=\frac{\partial L}{\partial \dr}(r(t_{n+1}),\dr(t_{n+1}^+))\, ,
\]
and analogously for $\frac{\partial L}{\partial \dr}(t_{n})$. Differentiating with respect to $t_n$ the relations $r(t_{n+1};t_n,t_{n+1})=R(t_{n+1})$ and $r(t_{n};t_n,t_{n+1})=R(t_n)$ we get
\[
\frac{\partial r}{\partial t_n}(t_{n+1}) =0, \qquad \dr(t_n^+)+\frac{\partial r}{\partial t_n}(t_{n})=\dR(t_n)\, .
\]
Hence we have
\[
\partial_{t_n} h(t_n,t_{n+1}) = -L(r(t_n),\dr(t_n^+)) +   \frac{\partial L}{\partial \dr}(t_n)(\dr(t^+_{n})-\dR(t_n)) 
\]
and, remembering the expression of $L$ in \eqref{def_L} and that $r(t_n)=R(t_n)$,
\begin{equation}
  \label{def_pt0}
    \partial_{t_n} h(t_n,t_{n+1}) =\frac{1}{2}\dr^2(t_{n}^+)+\frac{c^2}{2R^2(t_n)}-\dr(t^+_{n})\dR(t_n)  \, .
\end{equation}
Analogously, one can get
\[
\partial_{t_{n+1}} h(t_n,t_{n+1}) =-\frac{1}{2}\dr^2(t_{n+1}^-)-\frac{c^2}{2R^2(t_{n+1})}+\dr(t^-_{n+1})\dR(t_{n+1})\, .
\]

Therefore we conclude that if a sequence $(t_n)$ satisfies
\begin{equation}
  \label{d_EL}
\partial_1 h(t_n,t_{n+1}) + \partial_2 h(t_{n-1},t_{n}) =0 \quad\mbox{for every }n\in\ZZ\, ,
\end{equation}
where $\partial_1$ and $\partial_2$ denote differentiation with respect to the first and the second argument respectively, then 
\begin{equation*}
  \begin{aligned}
    \frac{1}{2}\dr^2(t_{n}^+)-\dr(t^+_{n})\dR(t_n) &= \frac{1}{2}\dr^2(t_{n}^-)-\dr(t^-_{n})\dR(t_{n}) \Leftrightarrow \\
    \frac{1}{2}\dr^2(t_{n}^+)-\dr(t^+_{n})\dR(t_n)+\frac{1}{2}\dR^2(t_n) &= \frac{1}{2}\dr^2(t_{n}^-)-\dr(t^-_{n})\dR(t_{n}) +\frac{1}{2}\dR^2(t_n)\Leftrightarrow \\
   \frac{1}{2}(\dr(t_{n}^+)-\dR(t_n))^2   &=\frac{1}{2}(\dr(t_{n}^-)-\dR(t_n))^2
    \end{aligned}
\end{equation*}
from which we get the bouncing condition \eqref{normcondr} using that $\dr(t_{n}^+)<\dR(t_n)$ and $\dr(t_{n}^-)>\dR(t_n)$ (see \eqref{cons-bound}). Conversely, a sequence satisfying \eqref{normcondr} also satisfies \eqref{d_EL}.  
Using also Remark \ref{rem-unire}, we have proved the following
\begin{Pro} \label{prima-per-h}
A sequence $(t_n,\dr(t_{n}^+;t_n,t_{n+1}))$, with $(t_{n+1}-t_n)$ satisfying \eqref{hp_dirsol1}-\eqref{hp_dirsol3} and $r(t;t_n,t_{n+1})$ being the solution to \eqref{dirich_r} given by Proposition \ref{dirich_sol}, defines a bouncing solution with angular momentum $c$ if and only if for every $n\in\ZZ$
\[
\partial_1 h^c(t_n,t_{n+1}) + \partial_2 h^c(t_{n-1},t_{n}) =0
\]
for the function $h^c(t_n,t_{n+1})$ defined in \eqref{genfun}.
\end{Pro}

We are now ready to give an explicit expression to the function $h^c$ in terms of the times of bouncing. 

\begin{Pro} \label{forma-h}
Fixed a value $c>0$ for the angular momentum, let $t_0,t_1$ be two consecutive bouncing times and $R_0:=R(t_0)$ and $R_1:= R(t_1)$ be the corresponding radii of the breathing circle. The generating function defined in \eqref{genfun} computed along solutions to \eqref{dirich_r} defined in \eqref{def-r-finale} has the form
\begin{equation}\label{fungen-finale}
h(t_0,t_1) = \frac{1}{2}(t_1-t_0)\, A(t_0,t_1)  + c\, \arctan \left(\frac{c(t_1-t_0)}{\sqrt{R_0^2R_1^2-c^2(t_1-t_0)^2}}\right)
\end{equation}
where $A(t_0,t_1)$ is defined in \eqref{def-AB-finale}.
  \end{Pro}
\begin{proof}
Using \eqref{energy} for solutions to \eqref{dirich_r} it holds
\[      
L^c(r,\dr) = \frac{1}{2}\dr^2-\frac{c^2}{2r^2} = E-\frac{c^2}{2r^2} =\frac{A(t_0,t_1)}{2}-\frac{c^2}{r^2}.
\]
Hence
\[
h^c(t_0,t_1)= \int_{t_0}^{t_1} L^c(r,\dr)dt =\frac{1}{2}(t_1-t_0)\, A(t_0,t_1) -c \int_{t_0}^{t_1}\frac{c}{r^2(t;t_0,t_1)}dt\, .
\]
Moreover by \eqref{def-r-finale} for $\theta(t;t_0,t_1)$ and Remark \ref{rem-angolo}, we have
\[
\int_{t_0}^{t_1}\frac{c}{r^2(t;t_0,t_1)}dt = \delta \theta(t_0,t_1) = \pi - \arctan \left(\frac{c(t_1-t_0)}{\sqrt{R_0^2\, R_1^2-c^2(t_1-t_0)^2}}\right)
\]
and the proof is finished, since the generating function is defined up to an additive constant.
\end{proof}

We now argue on the other direction. Given the function $h(t_0,t_1)$ in \eqref{fungen-finale}, we show that when restricted to a suitable subset, it is a generating function of the bouncing motion of a particle inside the breathing circle $\DD_t$.

\begin{Pro} \label{proph-twist}
Let $\eps\in(0,1)$ be a fixed parameter and let
\[
\sigma := \min \left\{ \frac{\underline{R}}{2\, \| \dR \|}\, ,\, \frac{2\, \sqrt{1+\sqrt{1-\eps^2}}\, \underline{R}}{\sqrt{\| \frac{d^2}{dt^2} R^2 \|}} \right\} \quad \text{and} \quad c\in \left(0,\eps\, \frac{\underline{R}^2}{\sigma}\right).
\]
Let us consider the strip $\Omega = \{(t_0,t_1)\in\RR^2 \: :\: 0<t_1-t_0<\sigma \}$. Then the function $h:\Omega \rightarrow \RR$ defined in \eqref{fungen-finale} satisfies
\begin{itemize}
\item[(i)] $h\in C^2(\Omega)$;
\item[(ii)] $h(t_0+1,t_1+1) = h(t_0,t_1)$ for all $(t_0,t_1)\in \Omega$;
\item[(iii)] $\partial_{t_0t_1} h(t_0,t_1) <0$ for all $(t_0,t_1)\in \Omega$, and $\partial_{t_0t_1} h(t_0,t_1) \to -\infty$ as $(t_1-t_0)\to 0$.
\end{itemize}
\end{Pro}

\begin{proof}
Properties (i) and (ii) follow immediately from the regularity and the periodicity of $R(t)$, and from the definition of $h(t_0,t_1)$.

To prove (iii), by standard computations we obtain
\begin{equation} \label{der1-h}
\partial_{t_0} h(t_0,t_1) =  \frac{c^2}{2\, R_0^2} + \frac 12 \left( \frac{R_0^2+\sqrt{R_0^2 R_1^2-c^2(t_1-t_0)^2}}{R_0\, (t_1-t_0)} \right)^2  + \dR(t_0)\, \frac{R_0^2+\sqrt{R_0^2\, R_1^2-c^2(t_1-t_0)^2}}{R_0\, (t_1-t_0)}
\end{equation}
hence, using that $h\in C^2$,
\begin{equation} \label{der-sec-mista-h}
\partial_{t_0t_1} h(t_0,t_1) =  \left( \frac{R_0^2+\sqrt{R_0^2 R_1^2-c^2(t_1-t_0)^2}}{R_0\, (t_1-t_0)} + \dR(t_0) \right)\, \partial_{t_1}\, \left( \frac{R_0^2+\sqrt{R_0^2\, R_1^2-c^2(t_1-t_0)^2}}{R_0\, (t_1-t_0)} \right)\, .
\end{equation}
We can then conclude by the following estimates. First, if $\dR(t_0)\ge 0$ then  
\[
\frac{R_0^2+\sqrt{R_0^2 R_1^2-c^2(t_1-t_0)^2}}{R_0\, (t_1-t_0)} + \dR(t_0) >0\, .
\]
If instead $\dR(t_0)<0$, then we use $(t_1-t_0)<\sigma$ and write
\[
\frac{R_0^2+\sqrt{R_0^2 R_1^2-c^2(t_1-t_0)^2}}{R_0\, (t_1-t_0)} \ge \frac{R_0}{t_1-t_0} > \frac{2\, \| \dR \|\, R_0}{\underline R} > -\dR(t_0)
\]
so that the first term in $\partial_{t_0t_1} h(t_0,t_1)$ is positive in all cases.

For the second term, we have
\[
\partial_{t_1}\, \left( \frac{R_0^2+\sqrt{R_0^2\, R_1^2-c^2(t_1-t_0)^2}}{R_0\, (t_1-t_0)} \right) = 
 R_0^2\,\, \frac{R_1\, \dR(t_1) \, (t_1-t_0) - \sqrt{R_0^2\, R_1^2-c^2(t_1-t_0)^2} - R_1^2}{(t_1-t_0)^2\, \sqrt{R_0^2\, R_1^2-c^2(t_1-t_0)^2}}
\]
which is negative both if $\dR(t_1)\le 0$, as sum of non-positive terms with $R_1>0$, and if $\dR(t_1)>0$ since by $(t_1-t_0)<\sigma$ it holds
\[
R_1\, \dR(t_1) \, (t_1-t_0) < \frac{R_1\, \dR(t_1) \, \underline{R}}{2\|\dR\|} \le \frac{R_1^2}{2}\, .
\]
Finally, from \eqref{der-sec-mista-h} we find
\[
\partial_{t_0t_1} h(t_0,t_1) = -\frac{R_0 (R_0 +R_1)^2}{(t_1-t_0)^3} + o\left(\frac{1}{(t_1-t_0)^3}\right)
\]
as $(t_1-t_0)\to 0^+$. Hence (iii) is proved.
\end{proof}

Let $\Omega\subset \RR^2$ be the strip defined in Proposition \ref{proph-twist} and rewrite $h:\Omega \to \RR$ as
\[
h(t_0,t_1) =\frac{R^2_0+R^2_1 + 2\sqrt{R^2_0R^2_1-c^2(t_{1}-t_0)^2}}{2(t_{1}-t_0)}  + c\, \arctan \left(\frac{c(t_1-t_0)}{\sqrt{R_0^2R_1^2-c^2(t_1-t_0)^2}}\right)
\]
where $R_0:=R(t_0)$, $R_1:=R(t_1)$, and $c\in (0,\eps \frac{\underline{R}^2}{\sigma})$. As before let $\partial_1$ and $\partial_2$ denote the partial derivatives with respect to the first and second argument of a function, and let $\TT$ denote the one dimensional torus.

\begin{Pro}\label{def_imp_map}
Define $\sigma_*:= max_{t\in\TT}\partial_1 h(t,t+\sigma)$. The equations
  \begin{equation}\label{def_imp}
    \left\{
    \begin{array}{l}
      K_0 = \partial_1 h(t_0,t_{1}) \\[0.1cm]
      K_{1} =- \partial_2 h(t_0,t_{1})
      \end{array}
    \right.
\end{equation}
  define implicitly a $C^k$ embedding $P:\TT\times (\sigma_*,+\infty)\rightarrow \TT\times\RR$, $P(t_0,K_0) = (t_{1},K_{1})$. Moreover, P is twist in the sense that
  \[
\frac{\partial{t_1}}{\partial{K_0}} <0 ,
  \]
  and exact symplectic in the sense that there exists a $C^k$ function $V:\TT\times \RR\rightarrow \RR$ such that
  \[
K_1dt_1-K_0dt_0 = dV(t_0,K_0).
  \]
 
\end{Pro}

\begin{proof}
 Using that $\partial_1 h(t_0,t_{1})\to +\infty$ as $(t_1-t_0)\to 0^+$ as can be shown from \eqref{der1-h} and Proposition \ref{proph-twist}-(iii), we can apply the implicit function theorem to the first of \eqref{def_imp} and get the $C^k$ function
  \begin{equation}\label{t1_imp}
    t_1 = t_1(t_0,K_0)
  \end{equation}
  for $(t_0,K_0)\in\TT\times(\sigma_*,+\infty)$. Inserting \eqref{t1_imp} into the second of \eqref{def_imp} we get the desired $C^k$ map $P$. To prove that it is injective we note that if $P(t_0,K_0)=P(t_0',K_0')$ then, using once again Proposition \ref{proph-twist}-(iii), the second of \eqref{def_imp} implies $t_0=t_0'$ that, substituted in the first gives also $K_0=K_0'$. \\
  By implicit differentiation of the first we get the twist condition:
  \[
\frac{\partial t_1}{\partial K_0} =\frac{1}{\partial_{12}h(t_0,t_1)} <0\, .
  \]
  Finally, if we define $V(t_0,K_0)=-h(t_0,t_1(t_0,K_0))$ we get
  \begin{align*}
    dV(t_0,K_0) &= \left(-h_1(t_0,t_1)-h_2(t_0,t_1)\frac{\partial t_1}{\partial t_0}\right)dt_0 - h_2(t_0,t_1)\frac{\partial t_1}{\partial K_0}dK_0 \\
    &= -K_0dt_0 +K_1\frac{\partial t_1}{\partial t_0}dt_0 + K_1\frac{\partial t_1}{\partial K_0}dK_0 = K_1dt_1-K_0dt_0.
  \end{align*}
 \end{proof}

\begin{Rem}
  It follows from the previous Proposition that a sequence $(t_n,K_n)$ is an orbit of the map $P$ if and only if, for every $n$, $(t_{n+1}-t_n)\in \Omega$ and
  \[
\partial_1 h(t_n,t_{n+1}) + \partial_2 h(t_{n-1},t_{n}) =0, \qquad  K_n=\partial_1 h(t_n,t_{n+1}). 
  \]
  \end{Rem}
We finally note that the orbits of the diffeomorphism $P$ give rise to bouncing solutions of the billiard map in the breathing circle $\DD_t$.

\begin{Pro}\label{prop-finale-twist}
  Let $(t_n,K_n)$ be an orbit of $P$, let us set $R_n:= R(t_n)$ and define $\dr(t_n^+)$ as
\begin{equation}\label{def_rdot}
\dr(t_n^+) := \dR(t_n) - \sqrt{\dR^2(t_n) - 2K_n - \frac{c^2}{R_n^2}}\, ,
\end{equation}
then $\{(t_n,\dr(t_n^+))\}$ represents a bouncing solution in the sense that $\{t_n\}$ is the sequence of bouncing times, and $\{\dr(t_n^+)\}$ is the sequence of radial velocities right after the bounce at time $t_n$ with corresponding trajectories between two consecutive bounces being the solutions to system \eqref{dirich_r} found in Proposition \ref{dirich_sol}. 	 
\end{Pro}

\begin{proof}
Using \eqref{def-AB-finale}, we write
\[
K_n = - \frac{1}{2} A(t_n,t_{n+1}) - \dR(t_n)\, \frac{ R_n^2+\sqrt{R_n^2\, R_{n+1}^2-c^2(t_{n+1}-t_n)^2}}{R_n\, (t_{n+1}-t_n)}\, .
\]
Since
\[
A(t_n,t_{n+1}) - \frac{c^2}{R_n^2} = \left( \frac{ R_n^2+\sqrt{R_n^2\, R_{n+1}^2-c^2(t_{n+1}-t_n)^2}}{R_n\, (t_{n+1}-t_n)} \right)^2
\]
we obtain that $\dr(t_n^+)$ defined in \eqref{def_rdot} can be written as
\begin{equation} \label{dr-tn+}
\begin{aligned}
\dr(t_n^+) =\, & \dR(t_n) - \sqrt{\left(\dR(t_n) +\frac{ R_n^2+\sqrt{R_n^2\, R_{n+1}^2-c^2(t_{n+1}-t_n)^2}}{R_n\, (t_{n+1}-t_n)} \right)^2} =\\[0.2cm]
=\, & - \frac{ R_n^2+\sqrt{R_n^2\, R_{n+1}^2-c^2(t_{n+1}-t_n)^2}}{R_n\, (t_{n+1}-t_n)}
\end{aligned}
\end{equation}
where we have used that
\[
\frac{ R_n^2+\sqrt{R_n^2\, R_{n+1}^2-c^2(t_{n+1}-t_n)^2}}{R_n\, (t_{n+1}-t_n)} > 2\, \|\dR\|\, \frac{ R_n^2+\sqrt{R_n^2\, R_{n+1}^2-c^2(t_{n+1}-t_n)^2}}{R_n\, \underline{R}} \ge 2\, \|\dR\| > - \dR(t_n)
\]
since $(t_{n+1}-t_n)<\sigma$. A straightforward computation shows that $\dr(t_n^+)$ is then equal to $\dr(t_n^+;t_n,t_{n+1})$, the velocity of the solution \eqref{def-r-finale} to system \eqref{dirich_r} found in Proposition \ref{dirich_sol}. 
 \end{proof}

\begin{Rem} \label{rem-good}
The definition of $\dr(t_n^+)$ in Proposition \ref{prop-finale-twist} is inspired by \eqref{def_pt0}. Notice that the two solutions of \eqref{def_pt0} are
\[
\dr_\pm:= \dR(t_n) \pm \sqrt{\dR^2(t_n) + 2\partial_{t_n}h(t_n,t_{n+1}) - \frac{c^2}{R_n^2}}
\]
which can be written as
\[
\dr_\pm:= \dR(t_n) \pm \left(\dR(t_n) +\frac{ R_n^2+\sqrt{R_n^2\, R_{n+1}^2-c^2(t_{n+1}-t_n)^2}}{R_n\, (t_{n+1}-t_n)} \right)\, .
\]
Now, since $(t_{n+1}-t_n)<\sigma$,
\[
\dr_+ = 2\dR(t_n) + \frac{ R_n^2+\sqrt{R_n^2\, R_{n+1}^2-c^2(t_{n+1}-t_n)^2}}{R_n\, (t_{n+1}-t_n)} >
\]
\[
> 2\dR(t_n) + 2\, \|\dR\|\, \frac{ R_n^2+\sqrt{R_n^2\, R_{n+1}^2-c^2(t_{n+1}-t_n)^2}}{R_n\, \underline{R}} \ge 2\dR(t_n) + 2\, \|\dR\| \ge \max\{ 0,\dR(t_n)\}.
\]
Hence $\dr_-$ is the only solution of \eqref{def_pt0} which may represent the radial velocity of a bouncing solution leaving the boundary. Moreover $\dr_-=-\dr_+ + 2\dR(t_n)$, hence $\dr_+$ and $\dr_-$ can be interpreted as the radial velocity before and after the bounce respectively.
\end{Rem}

\begin{Rem}\label{rem-vecchi-lavori}
For completeness we show that the map $P$ defined in Proposition \ref{def_imp_map} corresponds to the map $\cal{M}$ considered in \cite{sylvie1} with a different choice of variables. The variables used in \cite{sylvie1} are $(t,I)$, and the map ${\cal M}:(t_0,I_0) \mapsto (t_1,I_1)$ is implicitly given by
\[
\left\{
\begin{array}{l}
I_1 = -I_0 - 2\, R_1\, \dR(t_1) + \frac{c^2+I_0^2}{R_0^2}\, (t_1-t_0)\\[0.2cm]
(t_1-t_0) \left( \frac{c^2+I_0^2}{R_0^2}\, (t_1-t_0) - 2I_0 \right) = R_1^2-R_0^2
\end{array}
\right. 
\]
with $I_n = - R_n \dr(t_n^+)$ in our notations. The second equation is obtained by
\[
\frac{c^2+I_0^2}{R_0^2} = \frac{c^2}{R_0^2} + \dr^2(t_0^+) = A(t_0,t_1)
\]
and using \eqref{def-AB-finale} and \eqref{dr-tn+} for $\dr(t_0^+)$.
On the other hand, from Proposition \ref{prop-finale-twist} and \eqref{def_imp}
\[
\dr(t_{1}^+) = \dR(t_1) - \sqrt{\dR^2(t_1) -2\partial_2h(t_{0},t_1) - \frac{c^2}{R_1^2}}
\]
and, arguing as in the proof of Proposition \ref{proph-twist}, it holds
\[
\partial_2 h(t_{0},t_1) + \frac{c^2}{2R_1^2} = \dR(t_1)\, \frac{ R_1^2+\sqrt{R_{0}^2\, R_{1}^2-c^2(t_{1}-t_{0})^2}}{R_1\, (t_{1}-t_{0})} - \frac 12\, \left(\frac{ R_1^2+\sqrt{R_{0}^2\, R_{1}^2-c^2(t_{1}-t_{0})^2}}{R_1\, (t_{1}-t_{0})} \right)^2\, .
\]
Therefore, using
\[
\frac{ R_1^2+\sqrt{R_{0}^2\, R_{1}^2-c^2(t_{1}-t_{0})^2}}{R_1\, (t_{1}-t_{0})} > \frac{2\, R_1\, \|\dR\|}{\underline{R}} \ge 2\, \|\dR\| >\dR(t_1)
\]
for $(t_1-t_{0})<\sigma$, we have
\begin{equation} \label{dr-tn+-prima}
\dr(t_{1}^+) = \dR(t_1) - \left(\frac{ R_1^2+\sqrt{R_{0}^2\, R_{1}^2-c^2(t_{1}-t_{0})^2}}{R_1\, (t_{1}-t_{0})} - \dR(t_1)\right) = 2\, \dR(t_1) - \frac{ R_1^2+\sqrt{R_{0}^2\, R_{1}^2-c^2(t_{1}-t_{0})^2}}{R_1\, (t_{1}-t_{0})}\, . 
\end{equation}
Using \eqref{dr-tn+-prima} we get
\[
I_1 = - R_1 \dr(t_1^+) = -2 R_1\, \dR(t_1) + \frac{R_1^2+\sqrt{R_0^2 R_1^2 - c^2(t_1-t_0)^2}}{t_1-t_0}
\]
and again by \eqref{dr-tn+} for $\dr(t_0^+)$, we conclude
\[
I_1+I_0 = - 2 R_1\, \dR(t_1) + \frac{R_0^2+R_1^2+2\sqrt{R_0^2 R_1^2 - c^2(t_1-t_0)^2}}{t_1-t_0} =  - 2 R_1\, \dR(t_1) + (t_1-t_0)\, A(t_0,t_1)
\]
which is the first equation.
\end{Rem}

\section{Periodic and quasi-periodic orbits}\label{sec:per}

In this section we give the proof of Theorem \ref{Main1}. We begin with a preliminary result for the billiard map. Let $\eps\in (0,1)$ be a fixed parameter as in Definition \ref{def_R}.

\begin{Pro}\label{Mather_billiard}
  Suppose that $\sigma>2$ and fix $c\in \left(0,\eps\, \frac{\underline{R}^2}{\sigma}\right)$. Then, for every $1<\omega<\sigma-1$ 
  \begin{itemize}
      \item if $\omega=p/q\in\QQ$, then there exists a minimal orbit $(t_n,K_n)_{n\in\ZZ}$ of angular momentum $c$ of the billiard map such that $(t_{n+q},K_{n+q})=(t_n+p,K_n)$; 
      \item if $\omega\in\RR\setminus\QQ$, then there exists a minimal invariant set $M_\omega$ of rotation number $\omega$ made of orbits of angular momentum $c$. Moreover, $M_\omega$ is the graph of a Lipschitz function $u:\pi(M_\omega)\rightarrow\RR$, and $M_\omega$ is either an invariant curve or a Cantor set.
        \end{itemize}
  \end{Pro}

\begin{proof}
 Fix $\sigma>2$ and $c\in \left(0,\eps\, \frac{\underline{R}^2}{\sigma}\right)$. The function $h$ defined in \eqref{fungen-finale} is, by Proposition \ref{proph-twist}, a generating function when restricted to the set $\Omega = \{(t_0,t_1)\in\RR^2 \: :\: 0<t_1-t_0<\sigma \}$. 
Choose $\omega$ such that $1<\omega<\sigma-1$. Fix a positive number $\beta<\min\{\omega-1,\sigma-\omega-1\}$. By compactness there exists $\delta$ such that 
\[
h_{12}\leq\delta<0\quad\mbox{on } \Omega_\beta = \{(t_0,t_1)\in\RR^2 \: :\: \beta\leq t_1-t_0\leq\sigma-\beta \}.
\]
Hence, we can apply Lemma \ref{extension2} and find a generating function $\tilde{h}$ that coincide with $h$ on $\Omega_\beta$ and satisfies the hypothesis of Theorem \ref{Mather}. The function $\tilde{h}$ generates a diffeomorphism $\tilde{P}$ that coincide with the billiard map on some strip. Applying Theorem \ref{Mather}, for every $\tilde{\omega}\in\RR$ we find the periodic orbits and the invariant sets $M_{\tilde{\omega}}$ described in Theorem \ref{Mather}. These sets are made of orbits $(t_n,K_n)_{n\in\ZZ}$ of the diffeomorphism $\tilde{P}$ and become orbits of the billiard map if
\[
(t_n,t_{n+1})\in\Omega_\beta \quad \mbox{for every }n\in\ZZ.
\]
For $\tilde{\omega}=\omega$, by \eqref{stim-omega} we have that
\[
|t_{n+1}-t_n-\omega|\leq 1
\]
that implies, since $\omega>1$, that for every $n\in\ZZ$
\[
0<\omega-1\leq t_{n+1}-t_n\leq \omega+ 1.
\]
By the choice of $\beta$, for every $n\in\ZZ$,
  \[
\beta <\omega-1\leq t_{n+1}-t_n\leq \omega+ 1 < \sigma-\beta  
\]
that is $(t_n,t_{n+1})\in\Omega_\beta$ for every $n\in\ZZ$.
\end{proof}

Then Corollary \ref{MatherCor} immediately implies

\begin{Cor} \label{coro-subito}
   For each $1<\omega<\sigma-1$ there exist two functions $\phi,\eta:\RR\rightarrow\RR$ such that for every $\xi\in\mathbb{R}$
  \begin{align}
    &\phi(\xi+1) = \phi(\xi)+1, \quad \eta(\xi+1)=\eta(\xi), \label{non-serve}\\
    &S(\phi(\xi),\eta(\xi))=(\phi(\xi+\omega),\eta(\xi+\omega))\label{2mathercor}
  \end{align}
    where $\phi$ is monotone (strictly if $\omega\in\RR\setminus\mathbb{Q}$ ) and $\eta$ is of bounded variation.
  \end{Cor}
Let us now come to the
\begin{proof}[Proof of Theorem \ref{Main1}]
  Fix $c\in \left(0,\eps\, \frac{\underline{R}^2}{\sigma}\right)$. Consider $\omega\in\RR\setminus\QQ$ and the corresponding functions $\phi,\eta:\RR\rightarrow\RR$ given by Corollary \ref{coro-subito}.
  Denote by
  \[
x_\xi(t)=(r(t),\theta(t))_\xi 
  \]
  the bouncing solution with angular momentum $c$ which satisfies 
  \[
  r(\phi(\xi))=R(\phi(\xi))\, ,\quad \dr(\phi(\xi))= \dR(\phi(\xi)) - \sqrt{\dR^2(\phi(\xi)) - 2\eta(\xi) - \frac{c^2}{R_{\phi(\xi)}^2}}\, .
  \] 
Since the system is rotationally invariant, it is clear that fixing $c$, $r(\phi(\xi))$ and $\dr(\phi(\xi))$, is sufficient to uniquely determine the bouncing solution up to rotations. The value of $\theta(\phi(\xi))$ can be chosen freely. By the periodicity of $\phi,\eta$,
  \[
 (r(t),\theta(t))_{\xi+1} =(r(t-1),\theta(t-1))_\xi.
  \]
Finally, using \eqref{2mathercor} and  Proposition \ref{prop-finale-twist} we have
  \[  
(r(t),\theta(t))_{\xi+\omega}=  (r(t),\theta(t))_{\xi}.
\]
The last part of the statement follows from the definition of rotation number of a minimal orbit.
\end{proof}

\section{Chaotic motions} \label{sec:chaos}
In this section we prove the existence of chaotic motion for the billiard map inside the breathing circle $\DD_t$ with function $R(t)\in \CB$. In particular we prove the following version of Theorem \ref{Main2}.

\begin{The}\label{Main2prime}
Suppose that $R(t)\in\CB$. Then there exists $c_0>0$ such that for every $c\in(0,c_0)$ the map $P$ defined in Proposition \ref{def_imp_map} has positive topological entropy. More precisely, for every $c\in(0,c_0)$ there exists many $P$-invariant probability measures with positive metric entropy.   
  \end{The}

The idea of the proof is the following. First we extend $P$ to the whole cylinder as in the proof of Proposition \ref{Mather_billiard}. The key point is then to prove that there exists an open interval $\mathcal{I}\subset \RR$ such that for sufficiently small values of $c$, the extended map has no invariant curve with rotation number $\omega\in \mathcal{I}$. Hence, for irrational $\omega\in \mathcal{I}$ the Mather sets $M_\omega$ of Theorem \ref{Mather} are Cantor sets. Then Theorem \ref{Theo_forni} guarantees the existence of invariant probability measures with positive metric entropy for the extended map. The final step is to show that the extension has been made in such a way that these invariant measures are supported in the zone of the cylinder where the extended map coincide with $P$.

Let us first state and prove a series of technical lemmas. Let $R\in\CB$ and recall by Definition \ref{def_R} that $\sigma>4$. Moreover in this section we use the notations $\dR_{\bar t}$ and $\ddR_{\bar t}$ for $\dR(\bar t)$ and $\ddR(\bar t)$ respectively. Let us consider the set
\[
\Xi_R := \left\{ \omega\in (3,\sigma-1) \: : \: \frac{2\overline{R}^2}{\sigma^2}<\frac{2 \underline{R}^2}{(\omega+1)^2}- \| \dR\| \frac{2 \overline{R}}{\omega+1}<\frac{2\overline{R}^2}{(\omega-1)^2}
+\| \dR \|\frac{2\overline{R}}{\omega-1}
  <-\ddR_{\bar{t}}\underline{R}  \right\}.
\]

\begin{Lem}\label{stim_r}
If $R(t)\in \CB$ the set $\Xi_R$ is not empty and contains an open interval $\mathcal{I}$. 
\end{Lem}

  \begin{proof}
    Let us first note that
    \[
\frac{2\underline{R}^2}{\frac{2\overline{R}^2}{\sigma^2}+\| \dR\| \overline{R}} < \frac{\underline{R}^2}{\overline{R}^2}\sigma^2< \sigma^2
    \]
from which, by conditions (i) and (ii) of Definition \ref{def_R}, there exist $3<\omega^-<\omega^+<\sigma-1$ such that every $\omega\in(\omega^-,\omega^+)$ satisfies
    \begin{equation*}
1+\sqrt{\frac{2\overline{R}^2}{-\ddR_{\bar{t}}\underline{R}-\| \dR\| \overline{R}}}<\omega<-1+\sqrt{\frac{2\underline{R}^2}{\frac{2\overline{R}^2}{\sigma^2}+\| \dR\| \overline{R}}}
    \end{equation*}
    or, equivalently,
    \begin{equation}\label{cond_magic}
(\omega-1)^2>\frac{2\overline{R}^2}{-\ddR_{\bar{t}}\underline{R}-\| \dR\| \overline{R}} \quad\mbox{and}\quad (\omega+1)^2<\frac{2\underline{R}^2}{\frac{2\overline{R}^2}{\sigma^2}+\| \dR\| \overline{R}}.        
      \end{equation}
  Since $\omega>3$, using the first of \eqref{cond_magic}
  \[
\frac{2\overline{R}^2}{(\omega-1)^2}
+\| \dR \|\frac{2\overline{R}}{\omega-1}<\frac{2\overline{R}^2}{(\omega-1)^2}+\| \dR\| \overline{R}
  <-\ddR_{\bar{t}}\underline{R}
  \]
  that proves the third inequality in the definition of the set $\Xi_R$. Analogously, since $\omega>1$, using the second of \eqref{cond_magic}
  \[
\frac{2 \underline{R}^2}{(\omega+1)^2}- \| \dR\| \frac{2 \overline{R}}{\omega+1} > \frac{2 \underline{R}^2}{(\omega+1)^2}- \| \dR\|\overline{R}   >\frac{2\overline{R}^2}{\sigma^2}.
  \]
  that proves the first inequality in the definition of the set $\Xi_R$. The second inequality can be easily proved.  
 \end{proof}

\begin{Lem}\label{stim_inv}
    Let $c\in \left(0,\eps\, \frac{\underline{R}^2}{\sigma}\right)$ and $\mathcal{I}=(\omega^-,\omega^+)$ be the interval defined in Lemma \ref{stim_r}.  Let $\Gamma=\{(t,\gamma(t)) \: :\: t\in\TT \}$ be an invariant curve of the billiard map with rotation number $\omega\in\mathcal{I}$. Then
    \[
  \KK^-(\omega)+o(c)  \leq \gamma(t) \leq \KK^+(\omega)+o(c) \, ,
    \]
where
    \[
\KK^-(\omega)=\frac{2\underline{R}^2}{(\omega+1)^2}
 -\| \dR \|\frac{2\overline{R}}{\omega+1}
  ,
    \qquad    \KK^+(\omega)=\frac{2\overline{R}^2}{(\omega-1)^2}
+\| \dR \|\frac{2\overline{R}}{\omega-1}
\]
and $o(c)$ represents a function depending on $R,\omega,c,t_0,t_1$ that tends to zero uniformly for $\omega\in\mathcal{I}$ as $c\to 0^+$.
    \end{Lem}

\begin{proof}
    Let $(t_n,K_n)$ be an orbit of the billiard map with rotation number $\omega$ on the invariant curve $\Gamma$. From (\ref{def_imp}) and (\ref{der1-h}) a direct computation gives for the point $(t_0,K_0)$ of the orbit 
    \begin{align}\label{kazero}
      K_0 = \frac{1}{2}\left(\frac{R_0+R_1}{t_1-t_0} \right)^2+\dR(t_0) \left(\frac{R_0+R_1}{t_1-t_0} \right) - c^2 f(t_1,t_0,c)
     \end{align}
    where
\[
    f(t_1,t_0,c)= \frac{1}{2R_0}+\frac{R_0+\dR(t_0)(t_1-t_0)}{R_0^2R_1+R_0\sqrt{R_0^2 R_1^2-c^2(t_1-t_0)^2}}
\]
    can be bounded by a constant depending on $R$ and $c$ but not on $\omega$. Actually, from \eqref{stim-omega} and the fact that $\omega\in\mathcal{I}$,
\begin{equation}\label{usare-qui}
2<\omega^--1<\omega-1 <t_1(t_0,K_0,c)-t_0<\omega+1<\omega^++1
\end{equation}
so that $|t_1(t_0,K_0,c)-t_0|$ is uniformly bounded on every invariant curve with rotation number $\omega\in\mathcal{I}$ for $c$ fixed. Solving \eqref{kazero} for $(t_1-t_0)$ we get
\[
t_1-t_0=\frac{R_0+R_1}{\sqrt{\dR^2(t_0)+2(K_0+c^2f(t_1,t_0,c))}-\dR(t_0)}
\]
that used in \eqref{usare-qui} gives 
\[
\KK^-(\omega) \leq K_0+c^2f(t_1,t_0,c) \leq \KK^+(\omega)\, .
\]
Since this argument applies to all points of $\Gamma$ the proof is finished.
\end{proof}

\begin{Lem}\label{formula_a}
Let $\mathcal{I}=(\omega^-,\omega^+)$ be the interval defined in Lemma \ref{stim_r}. Suppose that there exists an invariant curve $\Gamma$ of the billiard map with rotation number $\omega\in\mathcal{I}$. Consider a point $(\bar{t},\bar{K})\in\Gamma$ such that
  \[
  \dR_{\bar{t}}=0.
  \]
 Let $t_1=t_1(\bar{t},K,c)$ and $t_{-1}=t_{-1}(\bar{t},K,c)$, and consider the function  
\[
a_c(\bar{t},K) := a(t_{-1}(\bar{t},K,c),\bar{t},t_1(\bar{t},K,c))
\]
with the notation given in Proposition \ref{Mather_prop_a}. Then
\[
a_c(\bar{t},K) = 2\sqrt{2K}\left(\ddR_{\bar{t}}+K\left(\frac{1}{R_0+R_{\bar{t}}}+\frac{1}{R_2+R_{\bar{t}}} \right) \right)+o(c),
\]
where $o(c)$ represents a function depending on $R,\omega,c,t_0,t_1$ that tends to zero uniformly for $\omega\in\mathcal{I}$ as $c\to 0$.
\end{Lem}
\begin{proof}
  Let $h(t,s)$ be the function defined in \eqref{fungen-finale} on the the strip $\Omega = \{(t,s)\in\RR^2 \: :\: 0<t-s<\sigma \}$, with $\eps$, $\sigma$ and $c$ given as in Proposition \ref{proph-twist}. Computations show that
\[
\begin{aligned}
\partial_{11} h (t,s) = &\, \ddR(t)\, \frac{R_t^2 + \sqrt{R_t^2 R_s^2 - c^2(s-t)^2}}{R_t\, (s-t)} + \frac{\dR^2(t)}{s-t} \left( 1 + \frac{c^2(s-t)^2}{R_t^2\sqrt{R_t^2 R_s^2 - c^2(s-t)^2}}\right) +\\[0.2cm]
& \, + 2\, \frac{R_t\, \dR(t)}{(s-t)^2}\, \left(1+ \frac{R_t^2 R_s^2}{R_t^2\sqrt{R_t^2 R_s^2 - c^2(s-t)^2}}\right) + \frac{A(t,s)}{s-t} + \frac{c^2(s-t)^2}{(s-t)^3\sqrt{R_t^2 R_s^2 - c^2(s-t)^2}}\\[0.4cm]
\partial_{22} h (t,s) = &\, \ddR(s)\, \frac{R_s^2 + \sqrt{R_t^2 R_s^2 - c^2(s-t)^2}}{R_s\, (s-t)} + \frac{\dR^2(s)}{s-t} \left( 1 + \frac{c^2(s-t)^2}{R_s^2\sqrt{R_t^2 R_s^2 - c^2(s-t)^2}}\right) +\\[0.2cm]
& \, - 2\, \frac{R_s\, \dR(s)}{(s-t)^2}\, \left(1+ \frac{R_t^2 R_s^2}{R_s^2\sqrt{R_t^2 R_s^2 - c^2(s-t)^2}}\right) + \frac{A(t,s)}{s-t} + \frac{c^2(s-t)^2}{(s-t)^3\sqrt{R_t^2 R_s^2 - c^2(s-t)^2}}
\end{aligned}
\]
where $R_t:=R(t)$ and $R_s:=R(s)$, and we recall from \eqref{def-AB-finale} that
\[
A(t,s) = \frac{R_t^2+R_s^2+2\sqrt{R_t^2 R_s^2 - c^2(s-t)^2}}{(s-t)^2}.
\]
Consider the function
\[
(t,K,c)\mapsto\, a_c(t,K) := \partial_{11}h (t,t_1) + \partial_{22}h (t_{-1},t)
\]
then
\[
\begin{aligned}
& a_c(t,K) = \, 2\, \ddR(t) \left(\dR(t)+ \frac{R_t^2+\sqrt{R_t^2 R_{1}^2-c^2(t_1-t)^2}}{R_t\, (t_1-t)} \right)+\\[0.2cm]
& \, + \frac{\dR^2(t)}{t-t_{-1}}\left( 1+ \frac{c^2(t-t_{-1})^2}{R_t^2\sqrt{R_t^2 R_{-1}^2-c^2(t-t_{-1})^2}}\right) +  \frac{\dR^2(t)}{t_1-t}\left( 1+ \frac{c^2(t_1-t)^2}{R_t^2\sqrt{R_t^2 R_{1}^2-c^2(t_1-t)^2}}\right) + \\[0.2cm]
&\, + \frac{A(t_{-1},t)}{t-t_{-1}} + \frac{c^2(t-t_{-1})^2}{(t-t_{-1})^3\sqrt{R_t^2 R_{-1}^2-c^2(t-t_{-1})^2}} + \frac{A(t,t_1)}{t_1-t} + \frac{c^2(t_1-t)^2}{(t_1-t)^3\sqrt{R_t^2 R_{1}^2-c^2(t_1-t)^2}} + \\[0.2cm]
&\, + 2\, \frac{R_t\, \dR(t)}{(t_1-t)^2}\, \left(1+ \frac{R_t^2 R_1^2}{R_t^2\sqrt{R_t^2 R_1^2 - c^2(t_1-t)^2}}\right) - 2\, \frac{R_t\, \dR(t)}{(t-t_{-1})^2}\, \left(1+ \frac{R_t^2 R_{-1}^2}{R_t^2\sqrt{R_t^2 R_{-1}^2 - c^2(t-t_{-1})^2}}\right) 
\end{aligned}
\]
where we have used that
\[
\begin{aligned}
& \frac{R_t^2 + \sqrt{R_t^2 R_{-1}^2 - c^2(t-t_{-1})^2}}{R_t\, (t-t_{-1})} = \dr(t^-;t_{-1},t) = \\[0.2cm]
& = -\dr(t^+;t,t_1)+2\dR(t) = \frac{R_t^2 + \sqrt{R_t^2 R_1^2 - c^2(t_{-1}-t)^2}}{R_t\, (t_{-1}-t)} +2\dR(t).
\end{aligned}
\] 
As in Lemma \ref{stim_inv}, since $t_1(t,K,c)-t$ and $t-t_{-1}(t,K,c)$ are uniformly bounded for $c$ fixed, we can write
\begin{equation*}
    \left\{
    \begin{array}{l}
      K = \frac{1}{2}\left(\frac{R_t+R_1}{t_1-t}\right)^2+\dR_t\frac{R_t+R_1}{t_1-t} +o(c) \\[0.1cm]
      K_{1} = \frac{1}{2}\left(\frac{R_t+R_1}{t_1-t}\right)^2-\dR_1\frac{R_t+R_1}{t_1-t} +o(c)
      \end{array}
    \right.
\end{equation*}
from which
\begin{equation*}
    \left\{
    \begin{array}{l}
      t_1-t = \frac{R_t+R_1}{\sqrt{\dR_t^2+2K}-\dR_t} +o(c) \\[0.1cm]
 t_0-t_{-1} = \frac{R_{-1}+R_t}{\sqrt{\dR_t^2+ 2K}+\dR_t}+o(c).     
      \end{array}
    \right.
\end{equation*}
Using this formulas in the expression of $a_c(t,K)$ for $t=\bar{t}$, since $\dR_{\bar{t}} =0$ we get from standard computations
\[
a_c(\bar{t},K) = 2\sqrt{2K}\left(\ddR_{\bar{t}}+K\left(\frac{1}{R_0+R_{\bar{t}}}+\frac{1}{R_2+R_{\bar{t}}}   \right)  \right) +o(c)
\]
where $o(c)$ is as in Lemma \ref{stim_inv}.
\end{proof}

\begin{Lem}\label{a_negative}
If $\dR_{\bar{t}}=0$ and $ \ddR_{\bar{t}}<-\frac{2\overline{R}^2}{\sigma^2\underline{R}}$ then 
\[
\alpha(\bar{t},K):= 2\sqrt{2K}\left(\ddR_{\bar{t}}+K\left(\frac{1}{R_0+R_{\bar{t}}}+\frac{1}{R_2+R_{\bar{t}}}   \right)  \right)<0
\] 
for every $K\in (\frac{2\overline{R}^2}{\sigma^2},-\ddR_{\bar{t}}\underline{R})$.
\end{Lem}

\begin{proof}
First note that
\[
\alpha(\bar{t},K) = 2\sqrt{2K}\left(\ddR_{\bar{t}}+K\left(\frac{1}{R_0+R_{\bar{t}}}+\frac{1}{R_2+R_{\bar{t}}}   \right)  \right)< 2\sqrt{2K}\left(\ddR_{\bar{t}}+\frac{K}{\underline{R}}   \right),
\]
from which we get $\alpha(\bar{t},K)<0$ for every $K\in (0,-\ddR_{\bar{t}}\underline{R})$. Moreover by the hypothesis on $\ddR_{\bar{t}}$ it holds $0<\frac{2\overline{R}^2}{\sigma^2}<-\ddR_{\bar{t}}\underline{R}$.
\end{proof}

We are now ready to extend the map $P$ to the cylinder $\TT\times \RR$. Fix $\sigma>4$. As in the proof of Proposition \ref{Mather_billiard}, for every $c\in (0,\varepsilon\frac{\underline{R}^2}{\sigma})$, the function $h$ defined in \eqref{fungen-finale} is, by Proposition \ref{proph-twist}, a generating function when restricted to the set $\Omega = \{(t_0,t_1)\in\RR^2 \: :\: 0<t_1-t_0<\sigma \}$. 
By Lemma \ref{stim_r}, we can fix $\omega\in\mathcal{I}=(\omega^-,\omega^+)\subset\Xi_R$. Fix a positive number $\beta<\min\{\omega^--1,\sigma-\omega^+-1\}$ and consider the set 
\[
\Omega_\beta = \{(t_0,t_1)\in\RR^2 \: :\: \beta\leq t_1-t_0\leq\sigma-\beta \}.
\]
Hence, we can apply Lemma \ref{extension2} and find a generating function $\tilde{h}$ that coincide with $h$ on $\Omega_\beta$ and satisfies the assumptions of Theorem \ref{Mather}. The function $\tilde{h}$ generates a diffeomorphism $\tilde{P}$ that coincide with the billiard map on some strip. 

The following result is crucial for the proof of Theorem \ref{Main2prime}. 
\begin{Pro}\label{prop_crash}
  Suppose that $R(t)\in\CB$.  
  Then, there exists $c_0$ such that for every $c\in(0,c_0)$ and $\omega\in\mathcal{I}\subset\Xi_R$ the extended map $\tilde{P}$ does not admit any invariant curve with rotation number $\omega$.
  \end{Pro}

\begin{proof}
Suppose by contradiction that for $c\to 0$, the map $\tilde{P}$ has an invariant curve $\Gamma$ with rotation number $\omega\in\mathcal{I}$. By \eqref{stim-omega}, for every orbit $\{(t_n,K_n)\}$ on $\Gamma$ we have that
\[
|t_{n+1}-t_n-\omega|\leq 1
\]
that implies, since $\omega>\omega^->3$, that for every $n\in\ZZ$
\[
0<\omega-1\leq t_{n+1}-t_n\leq \omega+ 1.
\]
By the choice of $\beta$, for every $n\in\ZZ$,
\[
\beta <\omega^--1<\omega-1\leq t_{n+1}-t_n\leq \omega+ 1 <\omega^++ 1 < \sigma-\beta  
\]
that is $(t_n,t_{n+1})\in\Omega_\beta$ for every $n\in\ZZ$. Hence the dynamics on the invariant curve is given by the billiard map $P$. By condition $(iii)$ in Definition \ref{def_R} there exists  a point $(\bar{t},\bar{K})\in\Gamma$ such that
\[
\dR_{\bar{t}}=0 \quad\mbox{and}\quad \ddR_{\bar{t}}<-\frac{2\overline{R}^2}{\sigma^2\underline{R}}.
\]
By Proposition \ref{Mather_prop_a} and Lemma \ref{formula_a},
\[
a_c(\bar{t},\bar{K}) = \alpha(\bar{t},\bar{K}) +o(c) = 2\sqrt{2\bar{K}}\left(\ddR_{\bar{t}}+\bar{K}\left(\frac{1}{R_0+R_{\bar{t}}}+\frac{1}{R_2+R_{\bar{t}}}   \right)  \right) +o(c)>0.
\]
However, from Lemma \ref{a_negative} we have $\alpha(\bar{t},K)<0$ for every $K\in (\frac{2\overline{R}^2}{\sigma^2},-\ddR_{\bar{t}}\underline{R})$ then, if 
\[
\bar{K}\in \left(\frac{2\overline{R}^2}{\sigma^2},-\ddR_{\bar{t}}\underline{R}\right)
\]
for $c$ sufficiently small, we obtain a contradiction and the proposition is proved. In fact, applying Lemmas \ref{stim_r} and \ref{stim_inv} we find
    \[
\bar{K} \leq \frac{2\overline{R}^2}{(\omega-1)^2}
+\| \dR \|\frac{2\overline{R}}{\omega-1} +o(c) < -\ddR_{\bar{t}}\underline{R} +o(c)
    \]
    and
   \[
  \bar{K} \geq \frac{2\underline{R}^2}{(\omega+1)^2}
 -\| \dR \|\frac{2\overline{R}}{\omega+1}+o(c) > \frac{2\overline{R}^2}{\sigma^2}+o(c),
 \]
where $o(c)$ represents a function that tends to zero for $c\to 0$ uniformly for $\omega\in\mathcal{I}$, and we are done.
\end{proof}

\begin{proof}[Proof of Theorem \ref{Main2prime}]
Consider the extended map $\tilde{P}$ for $c<c_0$, where $c_0$ is given as in Proposition \ref{prop_crash}. For every irrational $\omega\in\mathcal{I}\subset\Xi_R$, the Mather set $M_\omega$ is a Cantor set and there are no invariant curves with rotation number $\omega$. Hence, Theorem \ref{Theo_forni} gives, for every irrational $\omega\in\mathcal{I}$ the existence of a $\tilde{P}$-invariant measure $\mu_\omega$ with positive metric entropy arbitrarily close, in the sense specified in Theorem \ref{Theo_forni}, to the Mather set $M_\omega$. By the choice of the extension, as shown in Proposition \ref{prop_crash}, the Mather sets $M_\omega$ are contained in the zone of the cylinder where $\tilde{P}=P$. Hence there exist measures $\mu_\omega$ which are $P$-invariant.       
  \end{proof}

\appendix

\section{Proof of Proposition \ref{R_not_empty}} \label{technical}
It is clear that for every $k\in\mathbb{N}$, $\delta>0$, the function $R(t)$ is $C^2$, $1$-periodic and positive if $M>2\delta$.   
It is easily seen that
  \begin{equation}\label{Mm}
\underline{R}=M-2\delta\, , \quad \overline{R}=M+2\delta\, , \quad 
\quad \|\dR \|=2\pi\delta( k+1)
  \end{equation}
and
  \begin{equation}\label{Mm2}
\left\| \frac{d^2}{dt^2} R^2 \right\|\leq 8\pi^2\delta[(k^2+1)(M+3\delta)+2k\delta]\, .
  \end{equation}
Moreover, choosing $\bar{t}=\pi/2$,
\[
\dR(\bar{t})=0\, , \quad -\ddR(\bar{t})=4\pi^2\delta(k^2+1)\, . 
\]
Finally it is immediate from the hypothesis that $\delta<1$.

Let us start with the computation of $\sigma$. Using \eqref{Mm} we have that
\[
\sigma = (M-2\delta)\min \left\{ \frac{1}{4\pi\delta (k+1)}\, ,\, \frac{2\, \alpha}{\sqrt{\| \frac{d^2}{dt^2} R^2 \|}} \right\}.
\]
We note that $\| \frac{d^2}{dt^2} R^2 \|\geq 2\underline{R}(-\ddR(\bar{t}))=8\pi^2\delta(M-2\delta)(k^2+1) $. If
\begin{equation}\label{m1}
 M> \max\{ 34 \delta \, ,\,  2\delta + 216 \pi^2 (k+1)^2\}
\end{equation}
 using the fact that $\alpha\in (1,\sqrt{2})$ one can show that
\[
4\pi\delta (k+1)<\frac{\sqrt{8\pi^2\delta(M-2\delta)(k^2+1)}}{4}<   \frac{\sqrt{8\pi^2\delta(M-2\delta)(k^2+1)}}{2\alpha}=\frac{\sqrt{2\underline{R}(-\ddR(\bar{t}))}}{2\alpha}<\frac{\sqrt{\| \frac{d^2}{dt^2} R^2 \|}}{2\alpha}
\]
from which using \eqref{Mm2} and \eqref{m1}
\begin{equation}\label{stim_sigma}
\sigma = \frac{2(M-2\delta)\alpha}{\sqrt{\| \frac{d^2}{dt^2} R^2 \|}}>\frac{2(M-2\delta)\alpha}{\sqrt{8\pi^2\delta[(k^2+1)(M+3\delta)+2k\delta]} }>4\, .
\end{equation}
This gives condition $(i)$.

To prove that condition $(iii)$ holds we note that if
\begin{equation}\label{m3}
  M>2\delta\frac{4\pi^2\delta (k^2+1)+1}{4\pi^2\delta(k^2+1)-1}\, ,
\end{equation}
using that by hypothesis $4\pi^2\delta(k^2+1)-1>0$, we get
\[
-\ddR(\bar{t})=4\pi^2\delta(k^2+1)>\frac{M+2\delta}{M-2\delta}>\frac{2(M+2\delta)}{16(M-2\delta)}>\frac{2\overline{R}}{\sigma^2\underline{R}}
\]
since $\sigma>4$.

Let us now prove condition $(ii)$. By \eqref{m3} and the hypothesis $2\pi\delta (k+1)<1$ it holds
\[
-\ddR(\bar{t})\underline{R}-\| \dR\| \overline{R}=4\pi^2\delta(k^2+1)(M-2\delta)-2\pi\delta (k+1)(M+2\delta)>(M+2\delta)(1-2\pi\delta (k+1))>0.
\]

Since by \eqref{m1}
\[
M>8\pi^2\delta(k^2+1)
\]
then
\[
2<\sqrt{\frac{M}{2\pi^2\delta(k^2+1)}} <   \sqrt{\frac{2(M+2\delta)^2}{4\pi^2\delta (k^2+1)(M-2\delta)-2\pi\delta (k+1)(M+2\delta)}}=\sqrt{\frac{2\overline{R}^2}{-\ddR(\bar{t})\underline{R}-\| \dR\| \overline{R}}}
\]
that gives the first inequality in $(ii)$. To prove the second inequality, we note that the first inequality in \eqref{stim_sigma} gives
\[
\sqrt{\frac{2\underline{R}^2}{\frac{2\overline{R}^2}{\sigma^2}+\| \dR\| \overline{R}}}>  \sqrt{\frac{\alpha^2(M-2\delta)^4}{\pi\delta[2\pi(M+2\delta)^2((k^2+1)(M+3\delta)+2k\delta)+\alpha^2(k+1)(M+2\delta)(M-2\delta)^2]}} 
\]
so that we are done if
\begin{multline*}
1+\sqrt{\frac{2(M+2\delta)^2}{4\pi^2\delta (k^2+1)(M-2\delta)-2\pi\delta (k+1)(M+2\delta)}}  <\\
 -1+ \sqrt{\frac{\alpha^2(M-2\delta)^4}{\pi\delta[2\pi(M+2\delta)^2((k^2+1)(M+3\delta)+2k\delta)+\alpha^2(k+1)(M+2\delta)(M-2\delta)^2]}}
\end{multline*}
Looking at the asymptotic behaviour as $M\to +\infty$ of the left and right hand side of the previous inequality we find that the condition is equivalent to
\[
\sqrt{\frac{1}{\pi\delta[2\pi(k^2+1)-(k+1)]}M} <\sqrt{\frac{\alpha^2}{\pi\delta[2\pi(k^2+1)+\alpha^2(k+1)]}M}.
\]
Hence, the inequality is satisfied for $M$ large enough if
\begin{equation}\label{serve_eps}
\frac{1}{2\pi(k^2+1)-(k+1)}<\frac{\alpha^2}{2\pi(k^2+1)+\alpha^2(k+1)} \, .
\end{equation}
Since 
\[
\frac{1}{2\pi(k^2+1)-(k+1)}< \frac{1}{2\pi(k^2+1)-2\pi(k+1)}
\]
and
\[
\frac{\alpha^2}{2\pi(k^2+1)+2\pi\alpha^2(k+1)} < \frac{\alpha^2}{2\pi(k^2+1)+\alpha^2(k+1)}
\]
\eqref{serve_eps} is implied by
\[
\frac{1}{2\pi(k^2+1)-2\pi(k+1)}<\frac{\alpha^2}{2\pi(k^2+1)+2\pi\alpha^2(k+1)}
\]
or equivalently by
\[
(\alpha^2-1)k^2-2\alpha^2k-\alpha^2-1>0\, .
\]
Since $\alpha\in(1,\sqrt{2})$, it follows that a sufficient condition for \eqref{serve_eps} to hold is
\[
k>\frac{\alpha^2+\sqrt{2\alpha^4-1}}{\alpha^2-1}\, .
\]
This concludes the proof that for any $\eps\in (0,1)$, there exist $k,\delta,M$ such that the function 
\[
R_{k,\delta,M}(t) := M +\delta\sin\left(2\pi k t\right) +\delta\sin(2\pi t) 
\]
is in $\CB$. Moreover, for any 
\[
\eps\in \left(0,\sqrt{1-\frac{1}{(\pi-1)^2}}\right)
\]
we have $\alpha^2\in (\frac{\pi}{\pi-1},2)$, and the previous arguments can be repeated to show that there exist $\delta,M$ such that the function $R_{k,\delta,M}(t)$ is in $\CB$ for $k=1$ (it is enough to check \eqref{serve_eps} with $k=1$).
\qed

\section{Some results of Aubry-Mather theory}\label{sec:am}

In this section we gather the results from Aubry-Mather theory that are used in the paper. For the proofs we refer to  \cite{Aubry,bangert,forni,matherams,matherforni}.

Consider the cylinder $\Aa=\TT\times \RR$ and a strip $\Sigma=\TT\times(a,b)$ with $-\infty\leq a<b\leq +\infty$. Let  $S:\Sigma \rightarrow \Aa$, be a $C^2$-embedding and denote $S(x,y)=(\bar{x},\bar{y})$ and $S^n(x,y)=(x_n,y_n)$.

In the following, we will tacitly consider the lift of $S$ to the universal cover $\RR^2$ of $\Aa$ where $x\in\RR$, $\bar{x}(x+1,y)=\bar{x}(x,y)+1$ and $\bar{y}(x+1,y)=\bar{y}(x,y)$. With some abuse we will use the same notation for $S$ and its lift, and the correct interpretation should be clear from the context.

We suppose that $S$ is exact symplectic and twist. The exact symplectic condition requires the existence of a $C^2$ function $V:\Sigma\rightarrow \RR$ such that
\[
\bar{y} d \bar{x} -y dx = dV(x,y) \quad\mbox{in }\Sigma,
\]
and the (positive) twist condition reads
\[
\frac{\partial \bar{x}}{\partial y}>0 \quad\mbox{in }\Sigma. 
\]
A negative twist condition would give analogous results. If $\Sigma =\Aa$ we also suppose that $S$ preserves the ends of the cylinder that is
\[
\bar{y} \rightarrow \pm\infty \quad\mbox{as }y\rightarrow \pm\infty \mbox{ uniformly in } x,
\]
and twists each ends infinitely that is
\[
\bar{x}-x \rightarrow \pm\infty \quad\mbox{as }y\rightarrow \pm\infty \mbox{ uniformly in } x.
\]

Note that the exact symplectic condition implies that $S$ is orientation preserving and preserves the two-form $dy\wedge dx$.

For this class of maps, the following result is well known \cite{bangert,matherforni}. In the following we denote the partial derivative of $h$ with respect to the $i$-th variable by $h_i$.

\begin{Pro} \label{prop_h}
  Given $\Omega:=\left\{ (x,\bar{x})\in\RR^2 \: :\: \bar{x}(x,a)\leq \bar{x}\leq \bar{x}(x,b)\right\}$, there exists a $C^2$ function $h:\Omega \rightarrow\RR$ such that
  \begin{itemize}
  \item[(i)] $h(x+1,\bar{x}+1) = h(x,\bar{x})$ in $\Omega$,
  \item[(ii)] $h_{12}(x,\bar{x}) <0 $ in $\Omega$,
  \item[(iii)] for $(x,y)\in\Sigma$ we have $S(x,y) = (\bar{x},\bar{y})$ if and only if
    \[
    \left\{
    \begin{aligned}
      h_1(x,\bar{x})&=-y \\
      h_2(x,\bar{x})&=\bar{y}
    \end{aligned}
    \right.
\]
  \end{itemize}
 Conversely, for $\Omega':=\left\{ (x,\bar{x})\in\RR^2 \: :\: a' \leq \bar{x}-x\leq b'   \right\}$ let $h':\Omega' \rightarrow\RR$ be a $C^2$ function such that
  \begin{itemize}
  \item[(i)] $h'(x+1,\bar{x}+1) = h'(x,\bar{x})$ in $\Omega'$,
  \item[(ii)] $h_{12}'(x,\bar{x}) <0 $ in $\Omega'$,
 \end{itemize}   
    then, the equations
      \[
    \left\{
    \begin{aligned}
      h_1'(x,\bar{x})&=-y \\
      h_2'(x,\bar{x})&=\bar{y}
    \end{aligned}
    \right.
    \]
    define implicitly on $\Sigma':=\TT\times(-h_1'(x,\bar{x}+a'),-h_1'(x,\bar{x}+b'))$ a $C^2$ exact symplectic twist embedding $S':\Sigma'\rightarrow \mathbb{A}$.
\end{Pro} 
\begin{Rem}
If $\Sigma=\Aa$ the fact that $S$ preserves and twists each end infinitely implies that $\Omega=\RR^2$.
  The condition $h_{12}(x,\bar{x}) <0 $ is related to the twist condition. Actually, the twist implies that we can write $y = y(x,\bar{x})$ and one gets that
  \[
h_{12}(x,\bar{x}) = -\left(\frac{\partial \bar{x}}{\partial y}(x,y(x,\bar{x}))\right)^{-1}.
  \]
  \end{Rem}
The function $h$ (or $h'$) is called \emph{generating function} and gives an equivalent implicit definition of the diffeomorphism $S$. From this proposition one has that a sequence $(x_n,y_n)_{n\in\ZZ}$ such that $(x_n,y_n)\in\Sigma$ for every $n\in\ZZ$, is an orbit of $S$ if and only if for every $n\in\ZZ$ one has $(x_n,x_{n+1})\in \Omega$ and 
\begin{align}
\label{acca}  &h_2(x_{n-1},x_n) + h_1(x_n,x_{n+1})=0,  \\
  &y_n=-h_1(x_{n},x_{n+1})\nonumber.    
\end{align}

From now on, we consider the case $\Sigma =\Aa$. Actually, the following extension result (see for example \cite{matherforni,maro2}) guarantees that we can always extend an exact symplectic diffeomorphism defined on a strip to one defined on the cylinder.

\begin{Lem}
  \label{extension2}
  Let $h$ be a $C^2$ generating function defined on $\Omega=\left\{ (x,\bar{x})\in\RR^2 \: :\: a \leq \bar{x}-x\leq b   \right\}$ such that $h_{12}\leq\delta<0$ on $\Omega$. Then there exists a generating function $\tilde{h}$ defined on $\RR^2$ such that $h=\tilde{h}$ on $\Omega$ and $\tilde{h}\leq\delta<0$ on $\RR^2$. Moreover, $\tilde{h}=\frac{1}{2}(\bar{x}-x)^2$ on $\RR^2\setminus\Omega_\beta$, being $\Omega_\beta=\left\{ (x,\bar{x})\in\RR^2 \: :\: a-\beta \leq \bar{x}-x\leq b+\beta \right\}$. 
\end{Lem}

Let $\Sigma=\Aa$, we recall the variational characterisation of the orbits of $S$ in terms of the \emph{action}
\[
H_{\ell k}(x_\ell,\dots,x_k)=\sum_{n=\ell}^{k-1} h(x_n, x_{n+1})\, .
\]
It is well known that solutions of \eqref{acca} (and hence orbits of $S$) are in 1-1 correspondence with stationary points of $H_{\ell k}$ with respect to variations fixing the endpoints $x_\ell,x_k$. In the following we are interested in \emph{minimal orbits}, i.e. orbits $(x_n,y_n)_{n\in\ZZ}$ of $S$ such that for every pair of integers $h<k$ and for every sequence of real numbers $(x^*_n)_{\ell\leq n\leq k}$ such that $x_\ell^*=x_h$ and $x_k^*=x_k$, it holds
\begin{equation*}
 H_{\ell k}(x_\ell,\dots,x_k)    \leq H_{\ell k}(x^*_\ell,\dots,x^*_k).
 \end{equation*}
Moreover we recall that an orbit $(x_n,y_n)_{n\in\ZZ}$ of $S$ has \emph{rotation number} $\omega\in\RR$ if
\[
\lim_{n\to\infty}\frac{x_n}{n}=\omega.
\]
It is well known that minimal orbits are \emph{monotone}, that is only one of the following is satisfied:
\[
x_n<x_{n+1} \mbox{ for every } n\in\ZZ, \quad x_n=x_{n+1} \mbox{ for every } n\in\ZZ, \quad x_n>x_{n+1} \mbox{ for every } n\in\ZZ. 
\]
Moreover, if it has rotation number $\omega$, then it satisfies the following estimate for every $n,m\in\ZZ$:
\begin{equation}\label{stim-omega}
|x_n-x_m-(n-m)\omega|\leq 1.
\end{equation}
Finally we recall that an invariant set of $S$ is said to be minimal and with rotation number $\omega$ if it is made of minimal orbits with rotation number $\omega$, and that the term \emph{invariant curve of $S$} refers to a curve $\Gamma\subset\Sigma$ homotopic to $\{(x,y)\in\Aa \: : \:y=k, \mbox{ for some }k\in\RR  \}$ and such that $S(\Gamma) = \Gamma$. 

The following theorem gives the existence of minimal orbits with rotation number.

\begin{The}[\cite{bangert,matherams}]\label{Mather}
  Let $h:\RR^2\rightarrow\RR$ be a $C^2$ generating function such that
    \begin{itemize}
  \item[(i)] $h(x+1,\bar{x}+1) = h(x,\bar{x})$ in $\RR^2$,
  \item[(ii)] $h_{12}(x,\bar{x}) \leq\delta<0 $ in $\RR^2$
  \end{itemize}
    and let $S$ be the corresponding diffeomorphism. For a fixed $\omega\in\RR$
    \begin{itemize}
      \item if $\omega=p/q\in\QQ$, then there exists a minimal orbit $(x_n,y_n)_{n\in\ZZ}$ of $S$ such that $(x_{n+q},y_{n+q})=(x_n+p,y_n)$ 
      \item if $\omega\in\RR\setminus\QQ$, then there exists a minimal invariant set $M_\omega$ of rotation number $\omega$ such that $M_\omega$ is the graph of a Lipschitz function $u:\pi(M_\omega)\rightarrow\RR$. Moreover, $M_\omega$ is either an invariant curve or a Cantor set.
        \end{itemize}
\end{The}

The following corollary gives an equivalent interpretation of the result and has been proven in \cite{mathertop} (see also \cite{maro1}).

\begin{Cor}\label{MatherCor}
 For each $\omega\in \RR$ there exist two functions $\phi,\eta:\RR\rightarrow\RR$ such that for every $\xi\in \RR$
  \begin{align*}
    &\phi(\xi+1) = \phi(\xi)+1, \quad \eta(\xi+1)=\eta(\xi),\\
    &S(\phi(\xi),\eta(\xi))=(\phi(\xi+\omega),\eta(\xi+\omega))
  \end{align*}
  where $\phi$ is monotone (strictly if $\omega\in\RR\setminus\mathbb{Q}$ ) and $\eta$ is of bounded variation.
\end{Cor}

For irrational rotation numbers $\omega$, Theorem \ref{Mather} leaves open the possibility for the minimal set $M_\omega$ to be an invariant curve or not. To prove what is the case for a given $\omega$ is of fundamental importance to prove the existence of chaotic motion for the diffeomorphism $S$. We recall the following result by Forni.

Let us fix $\omega\in\RR\setminus\QQ$ and denote by $\sigma_\omega$ the unique $S$-invariant ergodic Borel probability measure supported on $M_\omega$.

\begin{The}[\cite{forni}]\label{Theo_forni}
  Let $S$ be a $C^2$ diffeomorphism of the cylinder $\Aa$ as in Theorem \ref{Mather}. Suppose that $S$ does not admit any invariant curve of rotation number $\omega$. Then there exists an $S$-invariant ergodic Borel probability measure $\mu_\omega$ with positive metric entropy. Moreover, $\mu_\omega$ can be chosen arbitrarily close to $\sigma_\omega$ in the sense of the weak topology on the space of compactly supported Borel probability measures on $\Aa$.  
\end{The}

Finally we recall a result to prove whether the set $M_\omega$ is an invariant curve or not.

\begin{Pro}[\cite{mather_glancing}]\label{Mather_prop_a}
  Let $S$ be a $C^2$ diffeomorphism of the cylinder $\Aa$ as in Theorem \ref{Mather} and let $\Gamma$ be an invariant curve of $S$. Then 
  \begin{itemize}
  \item[(i)] $\Gamma$ is a minimal set and each orbit on $\Gamma$ has the same rotation number;
 \item[(ii)] for any orbit $(x_n,y_n)$ on $\Gamma$ it holds
  \[
a(x_{n-1},x_n,x_{n+1}) := h_{22}(x_{n-1},x_n)+h_{11}(x_n,x_{n+1})>0\, , \quad \forall \, n\in \ZZ\, .
  \]
    \end{itemize}
\end{Pro}

\end{document}